\documentclass[a4paper]{amsart}

\def\mylabelonoff{off}
\def\allowdisbrkyesno{yes}
\def\numberingtheoremsectionyesno{yes}
\def\numberingequationsectionyesno{yes}
\def\pagesizeextendednormal{extended}

\def\reportudemathyesno{no}
\def\reportudemathnumber{SM-UDE-812}
\def\reportudemathyear{2017}
\def\reportudematheingang{\mydate}

\def\mytitle{A Global div-curl-Lemma for Mixed Boundary Conditions in Weak Lipschitz Domains\\
and a Corresponding Generalized $\Azs$-$\Ao$-Lemma in Hilbert Spaces}
\def\mytitlerepude{A Global div-curl-Lemma\\ 
for Mixed Boundary Conditions in Weak Lipschitz Domains\\
and a Corresponding Generalized $\Azs$-$\Ao$-Lemma in Hilbert Spaces}
\def\myshorttitle{A Global div-curl-Lemma}
\def\myauthorone{Dirk Pauly}
\def\myauthors{\myauthorone}
\def\myaddressone{Fakult\"at f\"ur Mathematik, Universit\"at Duisburg-Essen, Campus Essen, Germany}
\def\myemailone{dirk.pauly@uni-due.de}
\def\mykeywords{$\div$-$\curl$-lemma, compensated compactness, mixed boundary conditions, weak Lipschitz domains, Maxwell's equations}
\def\mysubjclass{35B27, 35Q61, 47B07, 46B50}
\def\mydate{\today}

\usepackage[mathscr]{eucal}
\usepackage[english]{babel}
\usepackage{a4,exscale,ifthen,amsfonts,amssymb,amsmath,amscd,graphicx,color}
\usepackage{nicefrac,tikz,fancyhdr,caption,array,multirow,multicol,booktabs,algorithm,algorithmic}
\usepackage[all]{xy}

\ifthenelse{\equal{\mylabelonoff}{on}}
{\newcommand{\mylabel}[1]{\label{#1}\fbox{{\sf #1}}}}
{\newcommand{\mylabel}[1]{\label{#1}}}
\ifthenelse{\equal{\allowdisbrkyesno}{yes}}
{\allowdisplaybreaks}
{}

\ifthenelse{\equal{\pagesizeextendednormal}{extended}}
{\setlength{\textwidth}{16cm}
\setlength{\textheight}{22cm}
\setlength{\oddsidemargin}{-0.2cm}
\setlength{\evensidemargin}{-0.2cm}}
{}
\ifthenelse{\equal{\numberingequationsectionyesno}{yes}}
{\numberwithin{equation}{section}}
{}

\newcommand{\ovl}[1]{\overline{#1}}

\newcommand{\cfp}{c_{\mathsf{f,p}}}

\newcommand{\cm}{c_{\mathsf{m}}}

\newcommand{\set}[2]{\{#1\,:\,#2\}}
\newcommand{\setb}[2]{\big\{#1\,:\,#2\big\}}

\ifthenelse{\equal{\numberingtheoremsectionyesno}{yes}}
{\newtheorem{lem}{Lemma}[section]}
{\newtheorem{lem}{Lemma}}

\newtheorem{theo}[lem]{Theorem}
\newtheorem{cor}[lem]{Corollary}
\newtheorem{rem}[lem]{Remark}

\newenvironment{acknow}{{\vspace*{1cm}\noindent\sc Acknowledgements }}{}

\newcommand{\om}{\Omega}

\newcommand{\ga}{\Gamma}
\newcommand{\gat}{\ga_{\mathsf{t}}}
\newcommand{\gan}{\ga_{\mathsf{n}}}

\newcommand{\eps}{\epsilon}

\newcommand{\reals}{\mathbb{R}}

\newcommand{\rt}{\reals^{3}}
\newcommand{\rN}{\reals^{N}}

\newcommand{\rttt}{\reals^{3\times3}}

\newcommand{\ot}{\leftarrow}
\newcommand{\hookto}{\hookrightarrow}
\newcommand{\impl}{\Rightarrow}

\newcommand{\equi}{\Leftrightarrow}
\newcommand{\qequi}{\quad\equi\quad}

\DeclareMathOperator{\id}{id}
\DeclareMathOperator{\sym}{sym}

\DeclareMathOperator{\dev}{dev}

\DeclareMathOperator{\supp}{supp}
\DeclareMathOperator{\dist}{dist}
\DeclareMathOperator{\A}{A}

\DeclareMathOperator{\As}{\A^{*}}

\DeclareMathOperator{\cA}{\mathcal{A}}

\DeclareMathOperator{\cAs}{\cA^{*}}

\DeclareMathOperator{\p}{\partial}
\DeclareMathOperator{\na}{\nabla}

\DeclareMathOperator{\rot}{rot}
\DeclareMathOperator{\rotc}{\overset{\circ}{\rot}}

\DeclareMathOperator{\Rot}{Rot}

\DeclareMathOperator{\curl}{curl}

\DeclareMathOperator{\divergence}{div}
\renewcommand{\div}{\divergence}
\DeclareMathOperator{\divc}{\overset{\circ}{\div}}
\DeclareMathOperator{\Div}{Div}

\DeclareMathOperator{\ed}{d}
\DeclareMathOperator{\cd}{\delta}

\newcommand{\csymbol}{\mathsf{C}}
\newcommand{\cgen}[3]{\overset{#1}{\csymbol}{}^{#2}_{#3}}

\newcommand{\cic}{\cgen{\circ}{\infty}{}}
\newcommand{\cicgat}{\cgen{\circ}{\infty}{\gat}}
\newcommand{\cicgan}{\cgen{\circ}{\infty}{\gan}}

\newcommand{\ct}{\cgen{}{2}{}}

\newcommand{\cicom}{\cic(\om)}
\newcommand{\cicgatom}{\cicgat(\om)}
\newcommand{\cicganom}{\cicgan(\om)}

\newcommand{\lsymbol}{\mathsf{L}}
\newcommand{\lgen}[3]{\overset{#1}{\lsymbol}{}^{#2}_{#3}}

\newcommand{\lt}{\lgen{}{2}{}}

\newcommand{\li}{\lgen{}{\infty}{}}

\newcommand{\lteps}{\lgen{}{2}{\eps}}
\newcommand{\ltmu}{\lgen{}{2}{\mu}}

\newcommand{\ltom}{\lt(\om)}

\newcommand{\liom}{\li(\om)}

\newcommand{\ltepsom}{\lteps(\om)}
\newcommand{\ltmuom}{\ltmu(\om)}

\newcommand{\hsymbol}{\mathsf{H}}

\newcommand{\hgen}[3]{\overset{#1}{\hsymbol}{}^{#2}_{#3}}

\newcommand{\ho}{\hgen{}{1}{}}
\newcommand{\htwo}{\hgen{}{2}{}}

\newcommand{\hoc}{\hgen{\circ}{1}{}}

\newcommand{\hocgat}{\hgen{\circ}{1}{\gat}}
\newcommand{\hocgan}{\hgen{\circ}{1}{\gan}}

\newcommand{\hoom}{\ho(\om)}

\newcommand{\htom}{\htwo(\om)}

\newcommand{\hocom}{\hoc(\om)}

\newcommand{\hocgatom}{\hocgat(\om)}
\newcommand{\hocganom}{\hocgan(\om)}

\newcommand{\hmo}{\hgen{}{-1}{}}
\newcommand{\hmoom}{\hmo(\om)}
\newcommand{\hmt}{\hgen{}{-2}{}}
\newcommand{\hmtom}{\hmt(\om)}

\newcommand{\rsymbol}{\mathsf{R}}

\newcommand{\rgen}[3]{\overset{#1}{\rsymbol}{}^{#2}_{#3}}

\renewcommand{\r}{\rgen{}{}{}}

\newcommand{\R}{\Rgen{}{}{}}
\newcommand{\rz}{\rgen{}{}{0}}

\newcommand{\rc}{\rgen{\circ}{}{}}

\newcommand{\rcz}{\rgen{\circ}{}{0}}
\newcommand{\rcga}{\rgen{\circ}{}{\ga}}
\newcommand{\rcgat}{\rgen{\circ}{}{\gat}}
\newcommand{\rcgan}{\rgen{\circ}{}{\gan}}

\newcommand{\rcgatz}{\rgen{\circ}{}{\gat,0}}
\newcommand{\rcganz}{\rgen{\circ}{}{\gan,0}}

\newcommand{\rom}{\r(\om)}
\newcommand{\rcom}{\rc(\om)}
\newcommand{\rczom}{\rcz(\om)}

\newcommand{\rzom}{\rz(\om)}

\newcommand{\rcgaom}{\rcga(\om)}
\newcommand{\rcgatom}{\rcgat(\om)}
\newcommand{\rcganom}{\rcgan(\om)}

\newcommand{\rcgatzom}{\rcgatz(\om)}
\newcommand{\rcganzom}{\rcganz(\om)}

\newcommand{\dsymbol}{\mathsf{D}}

\newcommand{\dgen}[3]{\overset{#1}{\dsymbol}{}^{#2}_{#3}}

\renewcommand{\d}{\dgen{}{}{}}

\newcommand{\dz}{\dgen{}{}{0}}

\newcommand{\dc}{\dgen{\circ}{}{}}

\newcommand{\dcgat}{\dgen{\circ}{}{\gat}}
\newcommand{\dcgan}{\dgen{\circ}{}{\gan}}

\newcommand{\dcgatz}{\dgen{\circ}{}{\gat,0}}
\newcommand{\dcganz}{\dgen{\circ}{}{\gan,0}}

\newcommand{\dom}{\d(\om)}

\newcommand{\dcom}{\dc(\om)}
\newcommand{\dzom}{\dz(\om)}

\newcommand{\dcgatom}{\dcgat(\om)}
\newcommand{\dcganom}{\dcgan(\om)}

\newcommand{\dcgatzom}{\dcgatz(\om)}
\newcommand{\dcganzom}{\dcganz(\om)}

\newcommand{\harmsymbol}{\mathcal{H}}
\newcommand{\harmgen}[3]{\overset{#1}{\harmsymbol}{}^{#2}_{#3}}

\newcommand{\harm}{\harmgen{}{}{}}

\newcommand{\harmom}{\harm(\om)}

\newcommand{\hilbert}{\hsymbol}

\newcommand{\hilo}{\hilbert_{1}}
\newcommand{\hilt}{\hilbert_{2}}

\newcommand{\norm}[1]{|#1|}

\newcommand{\bnorm}[1]{\big|#1\big|}

\newcommand{\normltom}[1]{\norm{#1}_{\ltom}}

\newcommand{\scp}[2]{\langle#1,#2\rangle}

\newcommand{\bscp}[2]{\big\langle#1,#2\big\rangle}

\newcommand{\scpltom}[2]{\scp{#1}{#2}_{\ltom}}

\newcommand{\preprintudemath}[5]{
\thispagestyle{empty}
\Large
\begin{center}SCHRIFTENREIHE DER FAKULT\"AT F\"UR MATHEMATIK\end{center}
\vspace*{5mm}
\begin{center}#1\end{center}
\vspace*{5mm}
\begin{center}by\end{center}
\begin{center}#2\end{center}
\vspace*{5mm}
\begin{center}#3\hspace{80mm}#4\end{center}
\newpage
\thispagestyle{empty}
\vspace*{210mm}
Received: #5
\newpage
\addtocounter{page}{-2}
\normalsize}

\title[\sc\myshorttitle]{\Large\sf\mytitle}
\author{\myauthorone}
\address{\myaddressone}
\email[\myauthorone]{\myemailone}
\subjclass{\mysubjclass}
\date{\mydate}

\newcommand{\To}{\longrightarrow}
\newcommand{\wto}{\rightharpoonup}
\renewcommand{\cic}{\mathring\csymbol^{\infty}}
\newcommand{\ccq}[2]{\mathring\csymbol^{#1}_{#2}}
\newcommand{\cicq}[2]{\mathring\csymbol^{\infty,#1}_{#2}}
\newcommand{\ccqom}[2]{\ccq{#1}{#2}(\om)}
\newcommand{\cicqom}[2]{\cicq{#1}{#2}(\om)}
\renewcommand{\cicgat}{\cic_{\gat}}
\renewcommand{\cicgan}{\cic_{\gan}}
\renewcommand{\hoc}{\mathring\hsymbol^{1}}
\newcommand{\hmoc}{\mathring\hsymbol^{-1}}
\newcommand{\hmtc}{\mathring\hsymbol^{-2}}
\newcommand{\hmocom}{\hmoc(\om)}
\newcommand{\hmtcom}{\hmtc(\om)}
\renewcommand{\hocgat}{\hoc_{\gat}}
\renewcommand{\hocgan}{\hoc_{\gan}}
\newcommand{\hoces}{\hoc_{\emptyset}}
\newcommand{\hocesom}{\hoces(\om)}
\newcommand{\htc}{\mathring\hsymbol^2}

\renewcommand{\rc}{\mathring\rsymbol}
\renewcommand{\rcz}{\rc_{0}}
\renewcommand{\rcga}{\rc_{\ga}}
\renewcommand{\rcgat}{\rc_{\gat}}
\renewcommand{\rcgatz}{\rc_{\gat,0}}
\renewcommand{\rcgan}{\rc_{\gan}}
\renewcommand{\rcganz}{\rc_{\gan,0}}
\renewcommand{\dc}{\mathring\dsymbol}
\renewcommand{\dcgan}{\dc_{\gan}}
\renewcommand{\dcganz}{\dc_{\gan,0}}
\renewcommand{\dcgat}{\dc_{\gat}}
\renewcommand{\dcgatz}{\dc_{\gat,0}}
\renewcommand{\H}{\mathsf{H}}

\renewcommand{\R}{\mathcal{R}}
\DeclareMathOperator{\Az}{A_{0}}
\DeclareMathOperator{\Azs}{A_{0}^{*}}
\DeclareMathOperator{\cAz}{\mathcal{A}_{0}}
\DeclareMathOperator{\cAzs}{\mathcal{A}_{0}^{*}}
\DeclareMathOperator{\Ao}{A_{1}}
\DeclareMathOperator{\Aos}{A_{1}^{*}}
\DeclareMathOperator{\cAo}{\mathcal{A}_{1}}
\DeclareMathOperator{\cAos}{\mathcal{A}_{1}^{*}}
\DeclareMathOperator{\At}{A_{2}}
\DeclareMathOperator{\Ats}{A_{2}^{*}}

\newcommand{\grad}{\na}
\newcommand{\gradc}{\mathring\grad}
\renewcommand{\rotc}{\mathring\rot}
\renewcommand{\divc}{\mathring\div}
\newcommand{\cptemb}{\hookto\hspace*{-0.85em}\hookto}
\newcommand{\bbS}{\mathbb{S}}
\newcommand{\bbT}{\mathbb{T}}
\newtheorem{theorom}{Theorem}

\newtheorem{corrom}[theorom]{Corollary}
\newtheorem{remrom}[theorom]{Remark}

\begin{document}


\ifthenelse{\equal{\reportudemathyesno}{yes}}
{\preprintudemath{\mytitlerepude}{\myauthors}{\reportudemathnumber}{\reportudemathyear}{\reportudematheingang}}
{}


\begin{abstract}
We prove global and local versions of the so-called $\div$-$\curl$-lemma,
a crucial result in the homogenization theory of partial differential equations,
for mixed boundary conditions on bounded weak Lipschitz domains in 3D
with weak Lipschitz interfaces.
We will generalize our results using an abstract Hilbert space setting,
which shows corresponding results to hold
in arbitrary dimensions as well as for various differential operators.
The crucial tools and the core of our arguments are Hilbert complexes and related compact embeddings.
\end{abstract}

\vspace*{-10mm}
\maketitle
\tableofcontents


\section{Introduction}
\mylabel{introsec}

The classical $\div$-$\curl$-lemma by Murat \cite{murat1978} and Tartar \cite{tartar1979},
a famous and crucial result in the homogenization theory of partial differential equations
and often used for so-called compensated compactness, reads as follows:

\begin{theorom}[classical $\div$-$\curl$-lemma]
\mylabel{introlemzero}
Let $\om\subset\rt$ be an open set and
let $(E_{n}),(H_{n})\subset\ltom$ be two sequences bounded in $\ltom$ such that
both $(\widetilde{\curl}E_{n})$ and $(\widetilde{\div}H_{n})$ are relatively compact in $\hmoom$.
Then there exist $E,H\in\ltom$ as well as subsequences, again denoted by $(E_{n})$ and $(H_{n})$, 
such that the sequence of scalar products $(E_{n}\cdot H_{n})$ converges in the sense of distributions, i.e.,
$$\forall\,\varphi\in\cicom\qquad
\int_{\om}\varphi\,(E_{n}\cdot H_{n})
\to\int_{\om}\varphi\,(E\cdot H).$$
\end{theorom}

Here, $\hmoom$ denotes the dual space of $\hocom$ and the distributional extensions 
$$\widetilde{\curl}:\ltom\to\hmoom,\qquad
\widetilde{\div}:\ltom\to\hmoom$$
of $\curl$ and $\div$, respectively, are defined for $E\in\ltom$ by
\begin{align*}
\widetilde{\curl}\,E\,(\Phi)&:=\scpltom{\curl\Phi}{E},\quad
\Phi\in\hocom,\\
\widetilde{\div}\,E\,(\varphi)&:=-\scpltom{\na\varphi}{E},\quad
\varphi\in\hocom.
\end{align*}

We will prove a global version of the $\div$-$\curl$-lemma
stating that under certain (mixed tangential and normal) 
boundary conditions and (very weak) regularity assumptions on a domain $\om\subset\rt$,
see Section \ref{defsec}, the following holds:

\begin{theorom}[global $\div$-$\curl$-lemma]
\mylabel{introtheoone}
Let $\om\subset\rt$ be a bounded weak Lipschitz domain
with boundary $\ga$ and weak Lipschitz boundary parts $\gat$ and $\gan$.
Let $(E_{n})$ and $(H_{n})$ be two sequences bounded in $\ltom$,
such that $(\curl E_{n})$ and $(\div H_{n})$ are also bounded in $\ltom$
and $\nu\times E_{n}=0$ on $\gat$ and $\nu\cdot H_{n}=0$ on $\gan$. Then there exist subsequences,
again denoted by $(E_{n})$ and $(H_{n})$, such that $(E_{n})$, $(\curl E_{n})$ and $(H_{n})$, $(\div H_{n})$ 
converge weakly to $E$, $\curl E$ and $H$, $\div H$ in $\ltom$, respectively,
and the inner products converge as well, i.e.,
$$\int_{\om}E_{n}\cdot H_{n}\to\int_{\om}E\cdot H.$$
\end{theorom}

A local version similar to the classical $\div$-$\curl$-lemma from Theorem \ref{introlemzero}
(distributional like convergence for arbitrary domains and no boundary conditions needed) 
is then immediately implied.

\begin{corrom}[local $\div$-$\curl$-lemma]
\mylabel{introcortwo}
Let $\om\subset\rt$ be an open set.
Let $(E_{n})$ and $(H_{n})$ be two sequences bounded in $\ltom$,
such that $(\curl E_{n})$ and $(\div H_{n})$ are also bounded in $\ltom$. Then there exist subsequences,
again denoted by $(E_{n})$ and $(H_{n})$, such that $(E_{n})$, $(\curl E_{n})$ and $(H_{n})$, $(\div H_{n})$ 
converge weakly to $E$, $\curl E$ and $H$, $\div H$ in $\ltom$, respectively,
and the inner products converge in the distributional sense as well, i.e.,
for all $\varphi\in\cicom$ it holds
$$\int_{\om}\varphi\,(E_{n}\cdot H_{n})\to\int_{\om}\varphi\,(E\cdot H).$$
\end{corrom}

For details see Theorem \ref{div-rot-lem}, Corollary \ref{div-rot-lem-loc},
Theorem \ref{diveps-murot-lem}, and Theorem \ref{class-div-rot-lem}.

We will also generalize these results to a natural Hilbert complex setting. For this, let 
$$\Az\!:\!D(\Az)\subset\H_{0}\to\H_{1},\qquad
\Ao\!:\!D(\Ao)\subset\H_{1}\to\H_{2}$$
be two (possibly unbounded) densely defined and closed linear operators 
on three Hilbert spaces $\H_{0}$, $\H_{1}$, $\H_{2}$ with Hilbert space adjoints 
$$\Azs\!:\!D(\Azs)\subset\H_{1}\to\H_{0},\qquad
\Aos\!:\!D(\Aos)\subset\H_{2}\to\H_{1}.$$
Moreover, let the complex property $\Ao\Az=0$ be satisfied, i.e.,
$$R(\Az)\subset N(\Ao).$$
In Theorem \ref{gen-div-rot-lem}
we present our central result of this contribution which reads as follows:

\begin{theorom}[generalized $\div$-$\curl$-lemma: $\Azs$-$\Ao$-lemma]
\mylabel{introtheothree}
Let $D(\Ao)\cap D(\Azs)\cptemb\H_{1}$ be compact.
If $(x_{n})\subset D(\Ao)$ and $(y_{n})\subset D(\Azs)$ 
are two $D(\Ao)$-bounded resp. $D(\Azs)$-bounded sequences, 
then there exist $x\in D(\Ao)$ and $y\in D(\Azs)$ as well as subsequences, 
again denoted by $(x_{n})$ and $(y_{n})$, 
such that $(x_{n})$ and $(y_{n})$ converge weakly in $D(\Ao)$ and $D(\Azs)$ to $x$ and $y$, respectively, 
together with the convergence of the inner products
$$\scp{x_{n}}{y_{n}}_{\H_{1}}\to\scp{x}{y}_{\H_{1}}.$$
\end{theorom}

\begin{remrom}
\mylabel{introremfour}
The compact embedding $D(\Ao)\cap D(\Azs)\cptemb\H_{1}$ reads in Theorem \ref{introtheoone} as
$$\setb{E\in\ltom}{\curl E\in\ltom,\,\div E\in\ltom,\,\nu\times E|_{\gat}=0,\,\nu\cdot E|_{\gan}=0}\cptemb\ltom,$$
which is known as Weck's selection theorem, see Lemma \ref{weckstlem}.
\end{remrom}

In  Theorem \ref{gen-div-rot-lem-gen-theo} the latter theorem 
is even generalized to a distributional version as follows:

\begin{theorom}[generalized $\div$-$\curl$-lemma: generalized $\Azs$-$\Ao$-lemma]
\mylabel{introtheofive}
Let the ranges $R(\Az)$ and $R(\Ao)$ be closed and let $N(\Ao)\cap N(\Azs)$ be finite-dimensional.
Moreover, let $(x_{n}),(y_{n})\subset\H_{1}$ be two bounded sequences such that
$(\widetilde{\Ao}\,x_{n})$ and $(\widetilde{\Azs}\,y_{n})$ 
are relatively compact in $D(\Aos)'$ and $D(\Az)'$, respectively.
Then there exist $x,y\in\H_{1}$ as well as subsequences, again denoted by $(x_{n})$ and $(y_{n})$, 
such that $(x_{n})$ and $(y_{n})$ converge weakly in $\H_{1}$ to $x$ and $y$, respectively, 
together with the convergence of the inner products
$$\scp{x_{n}}{y_{n}}_{\H_{1}}\to\scp{x}{y}_{\H_{1}}.$$
\end{theorom}

Here, the distributional extensions 
$$\widetilde{\Ao}:\H_{1}\to D(\Aos)',\qquad
\widetilde{\Azs}:\H_{1}\to D(\Az)'$$
of $\Ao$ and $\Azs$, respectively, are defined for $x\in\H_{1}$ by
\begin{align*}
\widetilde{\Ao}\,x\,(\phi)&:=\scp{\Aos\phi}{x}_{\H_{1}},\quad
\phi\in D(\Aos),\\
\widetilde{\Azs}\,x\,(\varphi)&:=\scp{\Az\varphi}{x}_{\H_{1}},\quad
\varphi\in D(\Az).
\end{align*}

In Section \ref{appsec} we apply these results to various differential operators in 3D and ND,
appearing, e.g., in classical and generalized electro-magnetics, for the biharmonic equation,
in general relativity, for gravitational waves, and in the theory 
of linear elasticity and plasticity.
We obtain also an interesting additional version of the 
global $\div$-$\curl$-lemma, compare to Theorem \ref{class-div-rot-lem-imp}.

\begin{theorom}[alternative global $\div$-$\curl$-lemma]
\mylabel{introtheosix}
Let $\om\subset\rt$ be a bounded strong Lipschitz domain with trivial topology.
Moreover, let $(E_{n}),(H_{n})\subset\ltom$ be two bounded sequences such that either
$(\widehat{\curl}E_{n})$ and $(\widetilde{\div}H_{n})$ are relatively compact in $\hmocom$ and $\hmoom$, respectively,
or $(\widetilde{\curl}E_{n})$ and $(\widehat{\div}H_{n})$ are relatively compact in $\hmoom$ and $\hmocom$, respectively.
Then there exist $E,H\in\ltom$ as well as subsequences, again denoted by $(E_{n})$ and $(H_{n})$, 
such that $E_{n}$ and $H_{n}$ converge weakly in $\ltom$, respectively, 
together with the convergence of the inner products
$$\scp{E_{n}}{H_{n}}_{\ltom}\to\scp{E}{H}_{\ltom}.$$
\end{theorom}

Here, $\hmocom:=\hoom'$ and the distributional extensions 
$$\widehat{\curl}:\ltom\to\hmocom,\qquad
\widehat{\div}:\ltom\to\hmocom$$
of $\curl$ and $\div$, respectively, are defined for $E\in\ltom$ by
\begin{align*}
\widehat{\curl}\,E\,(\Phi)&:=\scpltom{\curl\Phi}{E},\quad
\Phi\in\hoom,\\
\widehat{\div}\,E\,(\varphi)&:=-\scpltom{\na\varphi}{E},\quad
\varphi\in\hoom.
\end{align*}

The $\div$-$\curl$-lemma, which serves as a central result in the theory of compensated compactness,
see the original papers by Murat \cite{murat1978} and Tartar \cite{tartar1979}
with crucial applications in \cite{coifmanlionsmeyersemmescompcomphardy}
or \cite{evans1990,struwe2008},
and its variants and extensions have plenty of important applications.
For an extensive discussion and a historical overview of the $\div$-$\curl$-lemma see \cite{tartar2009}.
More recent discussions can be found, e.g., in \cite{brianecasadodiazmurat2009,tartar2015}
as well as in \cite{chenliglobrig}
and in the nice paper \cite{waurick2018a} of Marcus Waurick.
The latter two contributions utilize a Hilbert/Banach space setting as well,
but from different perspectives.
In \cite{waurick2018a} Waurick achieved closely related
results using different methods and proofs, see Section \ref{sec-gen-divcurl-lem}.
Interesting applications to homogenization of partial differential equations 
have recently been given in \cite{waurick2018b}.
From our personal\footnote{The idea of this paper came up a few years ago in 2012,
when S\"oren Bartels asked the author about the $\div$-$\curl$-lemma and for a simpler proof.
Moreover, in 2016, the $\div$-$\curl$-lemma in a form similar to the one in this article
was subject of lots of discussions with Marcus Waurick, when he as well as the author were lecturing
Special Semester Courses on Maxwell's equations and related topics invited by Ulrich Langer
at the Johann Radon Institute for Computational and Applied Mathematics (RICAM) in Linz.} 
point of view, although the results of \cite{chenliglobrig,waurick2018a} are slightly more general,
our methods and proofs are easier and more canonical and hence give deeper insight into the underlying structure 
and the core of the main result and thus of all $\div$-$\curl$-type lemmas.

The $\div$-$\curl$-lemma is widely used in the theory of homogenization of (nonlinear) partial differential equations,
see, e.g., \cite{struwe2008}.
Compensated compactness has many important applications 
in nonlinear partial differential equations and calculus of variations, e.g., 
in the partial regularity theory of stationary harmonic maps, see, e.g., 
\cite{freiremuellerstruwe1998,evans1991,riviere2007}.
Numerical applications can be found, e.g., in \cite{bartels2010}.
It is further a crucial tool in the homogenization of stochastic partial differential equations,
especially with certain random coefficients, see, e.g.,
the survey \cite{alexanderian2015}
and the literature cited therein, e.g., \cite{glorianeukammotto2015}. 

Let us also mention that the $\div$-$\curl$-lemma
is particularly useful to treat homogenization of problems arising in plasticity,
see, e.g., a recent contribution on this topic \cite{roegerschweizer2017},
for which \cite{schweizer2017} provides the important key $\div$-$\curl$-lemma.
As in \cite{schweizer2017,roegerschweizer2017} $\hoom$-potentials are used, 
these contributions are restricted to smooth, e.g., $\ct$ or convex,
domains and to full boundary conditions.
This clearly shows that the more general and stronger $\div$-$\curl$-lemma results
presented in the contribution at hand are of great importance
and so far unknown to the community.
The same $\hoom$-detour as in \cite{schweizer2017,roegerschweizer2017} 
is used in the recent contribution \cite{kozonoyanagisawa2013} 
where $\div$-$\curl$-type lemmas are presented
which also allow for inhomogeneous boundary conditions.
This unnecessarily high regularity assumption of $\hoom$-fields
excludes results like
\cite{kozonoyanagisawa2013,schweizer2017,roegerschweizer2017}
to be applied to important applications 
which are stated, e.g., in Lipschitz domains.

Generally, for problems related to Maxwell's equations
the detour over $\hoom$ and using Rellich's selection theorem
instead of using Weck's selection theorem, see Lemma \ref{weckstlem},
seems to be the wrong way to deal with such equations.
Most of the arguments simply fail, and if not,
the results are usually limited to smooth domains and trivial topologies.
Mixed boundary conditions cannot be treated properly.
Since the early 1970's, see the original paper by Weck \cite{weckmax}
for Weck's selection theorem,
it is well-known, that the $\hoom$-detour is often not helpful
and does not lead to satisfying results.
Surprisingly, this fact appears to be unknown to a wider community.

\section{Definitions and Preliminaries}
\mylabel{defsec}

Let $\om\subset\rt$ be a bounded weak Lipschitz domain, 
see \cite[Definition 2.3]{bauerpaulyschomburgmcpweaklip} for details, with boundary $\ga:=\p\om$, 
which is divided into two relatively open weak Lipschitz subsets $\gat$ 
and $\gan:=\ga\setminus\ovl{\gat}$ (its complement),
see \cite[Definition 2.5]{bauerpaulyschomburgmcpweaklip} for details.
Note that strong Lipschitz (graph of Lipschitz functions) implies weak Lipschitz (Lipschitz manifolds)
for the boundary as well as the interface. 
Throughout this section we shall assume the latter regularity on $\om$ and $\gat$.

Recently, in \cite{bauerpaulyschomburgmcpweaklip}, 
Weck's selection theorem, also known as the Maxwell compactness property, has been shown
to hold for such bounded weak Lipschitz domains and mixed boundary conditions.
More precisely, the following holds:

\begin{lem}[Weck's selection theorem]
\label{weckstlem}
The embedding
$\rcgatom\cap\dcganom\cptemb\ltom$
is compact.
\end{lem}

For a proof see \cite[Theorem 4.7]{bauerpaulyschomburgmcpweaklip}. 
A short historical overview of Weck's selection theorem
is given in the introduction of \cite{bauerpaulyschomburgmcpweaklip},
see also the original paper \cite{weckmax} and 
\cite{picardcomimb,webercompmax,costabelremmaxlip,witschremmax,jochmanncompembmaxmixbc,leisbook}
for simpler proofs and generalizations.

Here the usual Lebesgue and Sobolev spaces are denoted by $\ltom$ and $\hoom$ as well as
\begin{align*}
\rom:=\setb{E\in\ltom}{\rot E\in\ltom},\quad 
\dom:=\setb{E\in\ltom}{\div E\in\ltom},
\end{align*}
where we prefer to write $\rot$ instead of $\curl$.
$\rom$ and $\dom$ are also written as $H(\rot,\om)$, $H(\curl,\om)$ 
and $H(\div,\om)$ in the literature. With the help of test functions and test vector fields
$$\cicgatom:=\setb{\varphi|_\om}{\varphi\in\cic(\rt),\,\dist(\supp\varphi,\gat)>0}$$
we define the closed subspaces
\begin{align}
\mylabel{defstark}
\hocgatom:=\overline{\cicgatom}^{\hoom},\quad 
\rcgatom:=\overline{\cicgatom}^{\rom},\quad
\dcganom:=\overline{\cicganom}^{\dom}
\end{align}
as closures of test functions and vector fields, respectively. 
If $\gat=\ga$ we skip the index $\ga$ and write
$$\cic(\om)=\cic_{\ga}(\om),\qquad
\hoc(\om)=\hoc_{\ga}(\om),\qquad
\rc(\om)=\rc_{\ga}(\om),\qquad
\dc(\om)=\dc_{\ga}(\om).$$
In \eqref{defstark} homogeneous scalar, tangential and normal traces 
on $\gat$ and $\gan$ are generalized.
For the pathological case $\gat=\emptyset$, we put
$$\hocesom
:=\hoom\cap\reals^{\bot_{\ltom}}
=\setb{u\in\hoom}{\int_{\om}u=0}$$
in order to still have a Poincar\'e estimate for $u\in\hocesom$.
Let us emphasize that our assumptions also allow for Rellich's selection theorem, i.e., the embedding
\begin{align}
\mylabel{rellichst}
\hocgatom\cptemb\ltom
\end{align}
is compact, see, e.g., \cite[Theorem 4.8]{bauerpaulyschomburgmcpweaklip}. 
By density we have the two rules of integration by parts
\begin{align}
\mylabel{partintna}
\forall\,u&\in\hocgatom
&
\forall\,H&\in\dcganom
&
\scpltom{\na u}{H}&=-\scpltom{u}{\div H},\\
\mylabel{partintrot}
\forall\,E&\in\rcgatom
&
\forall\,H&\in\rcganom
&
\scpltom{\rot E}{H}&=\scpltom{E}{\rot H}.
\end{align}
We emphasize that,
besides Weck's selection theorem, the resulting Maxwell estimates (Friedrichs/Poincar\'e type estimates), 
Helmholtz decompositions, closed ranges, continuous and compact inverse operators,
and an appropriate electro-magneto static solution theory
for bounded weak Lipschitz domains and mixed boundary conditions,
another important result has been shown in \cite{bauerpaulyschomburgmcpweaklip}.
It holds
\begin{align}
\label{defschwach}
\begin{split}
\hocgatom&= 
\setb{u\in\hoom}{\scpltom{\na u}{\Phi}=-\scpltom{u}{\div\Phi}\text{ for all }\Phi\in\cicganom},\\
\rcgatom&
=\setb{E\in\rom}{\scpltom{\rot E}{\Phi}=\scpltom{E}{\rot\Phi}\text{ for all }\Phi\in\cicganom},\\
\dcganom&
=\setb{H\in\dom}{\scpltom{\div H}{\varphi}=-\scpltom{H}{\na\varphi}\text{ for all }\varphi\in\cicgatom},
\end{split}
\end{align}
i.e., strong and weak definitions of boundary conditions coincide,
see \cite[Theorem 4.5]{bauerpaulyschomburgmcpweaklip}.
Furthermore, we define the closed subspaces of irrotational and solenoidal vector fields
\begin{align*}
\rzom:=\setb{E\in\rom}{\rot E=0},\quad 
\dzom:=\setb{E\in\dom}{\div E=0},
\end{align*}
respectively, as well as 
$$\rcgatzom:=\rcgatom\cap\rzom,\quad
\dcganzom:=\dcganom\cap\dzom.$$

A direct consequence of Lemma \ref{weckstlem} is the compactness of the unit ball in 
$$\harmom:=\rcgatzom\cap\dcganzom,$$
the space of so-called Dirichlet-Neumann fields. Hence $\harmom$ is finite-dimensional.
Another immediate consequence of Weck's selection theorem, Lemma \ref{weckstlem}, using a standard indirect argument,
is the so-called Maxwell estimate, i.e., there exists $\cm>0$ such that
\begin{align}
\mylabel{maxest}
\forall\,E&\in\rcgatom\cap\dcganom\cap\harmom^{\bot_{\ltom}}
&
\normltom{E}&\leq\cm\big(\normltom{\rot E}+\normltom{\div E}\big)
\intertext{or, equivalently,}
\forall\,E&\in\rcgatom\cap\dcganom
&
\normltom{E-\pi E}&\leq\cm\big(\normltom{\rot E}+\normltom{\div E}\big),
\end{align}
see \cite[Theorem 5.1]{bauerpaulyschomburgmcpweaklip},
where $\pi:\ltom\to\harmom$ denotes the $\ltom$-orthonormal projector onto
the Dirichlet-Neumann fields.
Recent estimates for the Maxwell constant $\cm$ can be found in \cite{paulymaxconst0,paulymaxconst1,paulymaxconst2}.
Analogously, Rellich's selection theorem \eqref{rellichst} shows the Friedrichs/Poincar\'e estimate
\begin{align}
\mylabel{fpest}
\exists\,\cfp>0\quad\forall\,u\in\hocgatom\qquad
\normltom{u}\leq\cfp\normltom{\na u},
\end{align}
see \cite[Theorem 4.8]{bauerpaulyschomburgmcpweaklip}.
By the projection theorem, applied to the densely defined and closed (unbounded) linear operator
$$\na:\hocgatom\subset\ltom\longrightarrow\ltom$$
with (Hilbert space) adjoint
$$\na^{*}=-\div:\dcganom\subset\ltom\longrightarrow\ltom,$$
where we have used \eqref{defschwach}, we get the simple Helmholtz decomposition
\begin{align}
\mylabel{helmsim}
\ltom
&=\na\hocgatom\oplus_{\ltom}\dcganzom,
\end{align}
see \cite[Theorem 5.3 or (13)]{bauerpaulyschomburgmcpweaklip},
which immediately implies
\begin{align}
\mylabel{helmsimtwo}
\rcgatom
&=\na\hocgatom\oplus_{\ltom}\big(\rcgatom\cap\dcganzom\big)
\end{align}
as $\na\hocgatom\subset\rcgatzom$. 
Here $\oplus_{\ltom}$ in the decompositions \eqref{helmsim} and \eqref{helmsimtwo}
denotes the orthogonal sum in the Hilbert space $\ltom$.
By \eqref{fpest}, the range $\na\hocgatom$ is closed in $\ltom$,
see also \cite[Lemma 5.2]{bauerpaulyschomburgmcpweaklip}.
Note that we call \eqref{helmsim} a simple Helmholtz decomposition, 
since the refined Helmholtz decomposition
$$\ltom=\na\hocgatom\oplus_{\ltom}\harmom\oplus_{\ltom}\rot\rcganom$$
holds as well, 
see \cite[Theorem 5.3]{bauerpaulyschomburgmcpweaklip},
where also $\rot\rcganom$ is closed in $\ltom$
as a consequence of \eqref{maxest}, 
see \cite[Lemma 5.2]{bauerpaulyschomburgmcpweaklip}.

\section{The div-rot-Lemma}
\mylabel{divrotsec}

From now on we use synonymously
the notion $\div$-$\curl$-lemma and $\div$-$\rot$-lemma.
Let $\om\subset\rt$ be a bounded weak Lipschitz domain with weak Lipschitz interfaces
as introduced in Section \ref{defsec}.

\begin{theo}[global $\div$-$\rot$-lemma]
\mylabel{div-rot-lem}
Let $(E_{n})\subset\rcgatom$ and $(H_{n})\subset\dcganom$ be two sequences bounded in $\rom$ and $\dom$, respectively.
Then there exist $E\in\rcgatom$ and $H\in\dcganom$ as well as subsequences, again denoted by $(E_{n})$ and $(H_{n})$, 
such that
\begin{itemize}
\item
$E_{n}\wto E$ in $\rcgatom$,
\item
$H_{n}\wto H$ in $\dcganom$,
\item
$\scpltom{E_{n}}{H_{n}}\to\scpltom{E}{H}$.
\end{itemize}
\end{theo}

\begin{proof}
We pick subsequences, again denoted by $(E_{n})$ and $(H_{n})$, 
such that $(E_{n})$ and $(H_{n})$ converge weakly in $\rcgatom$ and $\dcganom$ to
$E\in\rcgatom$ and $H\in\dcganom$, respectively. By the simple Helmholtz decomposition \eqref{helmsimtwo},
we have the orthogonal decomposition $\rcgatom\ni E_{n}=\na u_{n}+\tilde{E}_{n}$ with some $u_{n}\in\hocgatom$
and $\tilde{E}_{n}\in\rcgatom\cap\dcganzom$. Then $(u_{n})$ is bounded in $\hoom$ by orthogonality 
and the Friedrichs/Poincar\'e estimate \eqref{fpest}.
$(\tilde{E}_{n})$ is bounded in $\rom\cap\dom$ by orthogonality and $\rot\tilde{E}_{n}=\rot E_{n}$, $\div\tilde{E}_{n}=0$.
Hence, using Rellich's and Weck's selection theorems, i.e., \eqref{rellichst} and Lemma \ref{weckstlem},
there exist $u\in\hocgatom$ and $\tilde{E}\in\rcgatom\cap\dcganzom$
and we can extract two subsequences, again denoted by $(u_{n})$ and $(\tilde{E}_{n})$, 
such that $u_{n}\rightharpoonup u$ in $\hocgatom$ and $u_{n}\to u$ in $\ltom$
as well as $\tilde{E}_{n}\rightharpoonup\tilde{E}$ in $\rcgatom\cap\dcganzom$ 
and $\tilde{E}_{n}\to\tilde{E}$ in $\ltom$. 
We have $E=\na u+\tilde{E}$, giving the simple Helmholtz decomposition for $E$, 
as, e.g., for all $\varphi\in\cicom$
\begin{align*}
\scpltom{E}{\varphi}
&\ot\scpltom{E_{n}}{\varphi}
=\scpltom{\na u_{n}}{\varphi}
+\scpltom{\tilde{E}_{n}}{\varphi}
\to\scpltom{\na u}{\varphi}
+\scpltom{\tilde{E}}{\varphi}.
\end{align*}
Then by \eqref{partintna}
\begin{align*}
\scpltom{E_{n}}{H_{n}}
&=\scpltom{\na u_{n}}{H_{n}}
+\scpltom{\tilde{E}_{n}}{H_{n}}
=-\scpltom{u_{n}}{\div H_{n}}
+\scpltom{\tilde{E}_{n}}{H_{n}}\\
&\to-\scpltom{u}{\div H}
+\scpltom{\tilde{E}}{H}
=\scpltom{\na u}{H}
+\scpltom{\tilde{E}}{H}
=\scpltom{E}{H},
\end{align*}
completing the proof.
\end{proof}

\begin{cor}[local $\div$-$\rot$-lemma]
\mylabel{div-rot-lem-loc}
Let $(E_{n})\subset\rom$ and $(H_{n})\subset\dom$ be two sequences bounded in $\rom$ and $\dom$, respectively.
Then there exist $E\in\rom$ and $H\in\dom$ as well as subsequences, again denoted by $(E_{n})$ and $(H_{n})$, 
such that $E_{n}\wto E$ in $\rom$ and $H_{n}\wto H$ in $\dom$ together with
the distributional convergence
$$\forall\,\varphi\in\cicom\qquad\scpltom{\varphi\,E_{n}}{H_{n}}\to\scpltom{\varphi\,E}{H}.$$
\end{cor}

\begin{proof}
Let $\gat:=\ga$ and hence $\gan=\emptyset$.
$(\varphi\,E_{n})$ is bounded in $\rcgaom$
and $(H_{n})$ is bounded in $\dom$.
Theorem \ref{div-rot-lem} shows the assertion.
\end{proof}

\begin{rem}
\mylabel{div-rot-rem}
We note that the boundedness of $(E_{n})$ and $(H_{n})$ 
in local spaces is sufficient for Corollary \ref{div-rot-lem-loc} to hold.
Hence, no regularity or boundedness assumptions on $\om$ are needed, i.e.,
Corollary \ref{div-rot-lem-loc} holds for an arbitrary open set $\om\subset\rt$.
Moreover, $\varphi\in\cicom$ may be replaced by $\varphi\in\ccqom{1}{}$ or even $\varphi\in\ccqom{0,1}{}$,
the space of Lipschitz continuous functions vanishing in a neighbourhood of $\ga$.
\end{rem}

\section{Generalizations}
\mylabel{gensec}

The idea of the proof of Theorem \ref{div-rot-lem} can be generalized.

\subsection{Functional Analysis Toolbox}

Let $\A\!:\!D(\A)\subset\H_{1}\to\H_{2}$ be a (possibly unbounded) 
closed and densely defined linear operator on two Hilbert spaces $\H_{1}$ and $\H_{2}$
with adjoint $\As\!:\!D(\As)\subset\H_{2}\to\H_{1}$.
Note $(\A^{*})^{*}=\ovl{\A}=\A$, i.e., $(\A,\As)$ is a dual pair.
By the projection theorem the Helmholtz type decompositions
\begin{align}
\mylabel{helm}
\H_{1}=N(\A)\oplus_{\H_{1}}\ovl{R(\As)},\quad
\H_{2}=N(\As)\oplus_{\H_{2}}\ovl{R(\A)}
\end{align}
hold, where we introduce the notation $N$ for the kernel (or null space)
and $R$ for the range of a linear operator.
We can define the reduced operators
\begin{align*}
\cA&:=\A|_{\ovl{R(\As)}}:D(\cA)\subset\ovl{R(\As)}\to\ovl{R(\A)},&
D(\cA)&:=D(\A)\cap N(\A)^{\bot_{\H_{1}}}=D(\A)\cap\ovl{R(\As)},\\
\cAs&:=\As|_{\ovl{R(\A)}}:D(\cAs)\subset\ovl{R(\A)}\to\ovl{R(\As)},&
D(\cAs)&:=D(\As)\cap N(\As)^{\bot_{\H_{2}}}=D(\As)\cap\ovl{R(\A)},
\end{align*}
which are also closed and densely defined linear operators.
We note that $\cA$ and $\cAs$ are indeed adjoint to each other, i.e.,
$(\cA,\cAs)$ is a dual pair as well. Now the inverse operators 
$$\cA^{-1}:R(\A)\to D(\cA),\qquad
(\cAs)^{-1}:R(\As)\to D(\cAs)$$
exist and are bijective, 
since $\cA$ and $\cAs$ are injective by definition.
Furthermore, by \eqref{helm} we have
the refined Helmholtz type decompositions
\begin{align}
\label{DacA}
D(\A)&=N(\A)\oplus_{\H_{1}}D(\cA),&
D(\As)&=N(\As)\oplus_{\H_{2}}D(\cAs)
\intertext{and thus we obtain for the ranges}
\label{RacA}
R(\A)&=R(\cA),&
R(\As)&=R(\cAs).
\end{align}

By the closed range theorem and the closed graph theorem we get immediately the following.

\begin{lem}
\label{poincarerange}
The following assertions are equivalent:
\begin{itemize}
\item[\bf(i)] 
$\exists\,c_{\A}\,\,\in(0,\infty)$ \quad 
$\forall\,x\in D(\cA)$ \qquad
$\norm{x}_{\H_{1}}\leq c_{\A}\norm{\A x}_{\H_{2}}$
\item[\bf(i${}^{*}$)] 
$\exists\,c_{\As}\in(0,\infty)$ \quad 
$\forall\,y\in D(\cAs)$ \quad\;
$\norm{y}_{\H_{2}}\leq c_{\As}\norm{\As y}_{\H_{1}}$
\item[\bf(ii)] 
$R(\A)=R(\cA)$ is closed in $\H_{2}$.
\item[\bf(ii${}^{*}$)] 
$R(\As)=R(\cAs)$ is closed in $\H_{1}$.
\item[\bf(iii)] 
$\cA^{-1}:R(\A)\to D(\cA)$ is continuous and bijective.
\item[\bf(iii${}^{*}$)] 
$(\cAs)^{-1}:R(\As)\to D(\cAs)$ is continuous and bijective.
\end{itemize}
In case that one of the latter assertions is true,
e.g., (ii), $R(\A)$ is closed, we have
\begin{align*}
\H_{1}
&=N(\A)\oplus_{\H_{1}}R(\As),
&
\H_{2}
&=N(\As)\oplus_{\H_{2}}R(\A),\\
D(\A)
&=N(\A)\oplus_{\H_{1}}D(\cA),
&
D(\As)
&=N(\As)\oplus_{\H_{2}}D(\cAs),\\
D(\cA)
&=D(\A)\cap R(\As),
&
D(\cAs)
&=D(\As)\cap R(\A),
\end{align*}
and 
\begin{align*}
\cA:D(\cA)\subset R(\As)\to R(\A),\quad
\cAs:D(\cAs)\subset R(\A)\to R(\As).
\end{align*}
\end{lem}

\begin{rem}
\label{remconstants}
For the ``best'' constants $c_{\A}$, $c_{\As}$ the following holds:
The Rayleigh quotients 
$$\frac{1}{c_{\A}}
:=\inf_{0\neq x\in D(\cA)}\frac{\norm{\A x}_{\H_{2}}}{\norm{x}_{\H_{1}}},\quad
\frac{1}{c_{\As}}
:=\inf_{0\neq y\in D(\cAs)}\frac{\norm{\As y}_{\H_{1}}}{\norm{y}_{\H_{2}}}$$
coincide, i.e., $c_{\A}=c_{\As}\in(0,\infty]$.
\end{rem}

\begin{lem}
\mylabel{cptequi}
The following assertions are equivalent:
\begin{itemize}
\item[\bf(i)]
$D(\cA)\cptemb\hilo$ is compact.
\item[\bf(i${}^{*}$)]
$D(\cAs)\cptemb\hilt$ is compact.
\item[\bf(ii)]
$\cA^{-1}:R(\A)\to R(\As)$ is compact.
\item[\bf(ii${}^{*}$)]
$(\cAs)^{-1}:R(\As)\to R(\A)$ is compact.
\end{itemize}
If one of these assertions holds true, e.g., (i), $D(\cA)\cptemb\hilo$ is compact,
then the assertions of Lemma \ref{poincarerange} and Remark \ref{remconstants} 
hold with $c_{\A}=c_{\As}\in(0,\infty)$.
Especially, the Friedrichs/Poincar\'e type estimates hold,
all ranges are closed and the inverse operators 
$$\cA^{-1}:R(\A)\to R(\As),\quad
(\cAs)^{-1}:R(\As)\to R(\A)$$
are compact with norms 
$\bnorm{\cA^{-1}}_{R(\A),R(\As)}
=\bnorm{(\cAs)^{-1}}_{R(\As),R(\A)}
=c_{\A}$.
\end{lem}

\begin{proof}
As the other assertions are easily proved
or immediately clear by symmetry,
we just show that (i), i.e., the compactness of
$$D(\cA)=D(\A)\cap\ovl{R(\As)}\cptemb\hilo,$$ 
implies (i${}^{*}$) as well as Lemma \ref{poincarerange} (i).

(i)$\impl$Lemma \ref{poincarerange} (i):
For this we use a standard indirect argument.
If Lemma \ref{poincarerange} (i) were wrong,
there would exist a sequence $(x_{n})\subset D(\cA)$ with $\norm{x_{n}}_{\H_{1}}=1$
and $\A x_{n}\to0$. As $(x_{n})$ is bounded in $D(\cA)$
we can extract a subsequence, again denoted by $(x_{n})$,
with $x_{n}\to x\in\H_{1}$ in $\H_{1}$. Since $\cA$ is closed,
we have $x\in D(\cA)$ and $\A x=0$, hence $x\in N(\cA)=\{0\}$, 
in contradiction to $1=\norm{x_{n}}_{\H_{1}}\to\norm{x}_{\H_{1}}=0$.

(i)$\impl$(i${}^{*}$):
Let $(y_{n})\subset D(\cAs)$ be a bounded sequence.
Utilizing Lemma \ref{poincarerange} (i) and (ii) we obtain
$D(\cAs)=D(\As)\cap R(\cA)$ and thus $y_{n}=\A x_{n}$ 
with $(x_{n})\subset D(\cA)$, which is bounded in $D(\cA)$ 
by Lemma \ref{poincarerange} (i). Hence we may extract a subsequence,
again denoted by $(x_{n})$, converging in $\H_{1}$. 
Therefore with $x_{n,m}:=x_{n}-x_{m}$ and $y_{n,m}:=y_{n}-y_{m}$ we see
\begin{align*}
\norm{y_{n,m}}_{\H_{2}}^2
&=\bscp{y_{n,m}}{\A(x_{n,m})}_{\H_{2}}
=\bscp{\As(y_{n,m})}{x_{n,m}}_{\H_{1}}
\leq c\,\norm{x_{n,m}}_{\H_{1}},
\end{align*}
and hence $(y_{n})$ is a Cauchy sequence in $\H_{2}$.
\end{proof}

Now, let $\Az\!:\!D(\Az)\subset\H_{0}\to\H_{1}$ and $\Ao\!:\!D(\Ao)\subset\H_{1}\to\H_{2}$ 
be (possibly unbounded) closed and densely defined linear operators 
on three Hilbert spaces $\H_{0}$, $\H_{1}$, and $\H_{2}$
with adjoints $\Azs\!:\!D(\Azs)\subset\H_{1}\to\H_{0}$
and $\Aos\!:\!D(\Aos)\subset\H_{2}\to\H_{1}$
as well as reduced operators $\cAz$, $\cAzs$, and $\cAo$, $\cAos$.
Furthermore, we assume the sequence or complex property of $\Az$ and $\Ao$, 
that is, $\Ao\Az=0$, i.e.,
\begin{align}
\mylabel{sequenceprop}
R(\Az)\subset N(\Ao).
\end{align}
Then also $\Azs\Aos=0$, i.e., $R(\Aos)\subset N(\Azs)$.
From the Helmholtz type decompositions \eqref{helm} for $\A=\Az$ and $\A=\Ao$ we get in particular
\begin{align}
\label{helmappclone}
\H_{1}&=\ovl{R(\Az)}\oplus_{\H_{1}}N(\Azs),
&
\H_{1}&=\ovl{R(\Aos)}\oplus_{\H_{1}}N(\Ao),
\end{align}
and the following result for Helmholtz type decompositions:

\begin{lem}
\label{helmrefined}
Let $N_{0,1}:=N(\Ao)\cap N(\Azs)$.
The refined Helmholtz type decompositions
\begin{align}
\mylabel{helmrefinedzero}
N(\Ao)&=\ovl{R(\Az)}\oplus_{\H_{1}}N_{0,1},
&
D(\Ao)&=\ovl{R(\Az)}\oplus_{\H_{1}}\big(D(\Ao)\cap N(\Azs)\big),
&
R(\Az)&=R(\cAz),\\
\mylabel{helmrefinedone}
N(\Azs)&=\ovl{R(\Aos)}\oplus_{\H_{1}}N_{0,1},
&
D(\Azs)&=\ovl{R(\Aos)}\oplus_{\H_{1}}\big(D(\Azs)\cap N(\Ao)\big),
&
R(\Aos)&=R(\cAos),
\end{align}
and
\begin{align}
\mylabel{helmrefinedtwo}
\H_{1}&=\ovl{R(\Az)}\oplus_{\H_{1}}N_{0,1}\oplus_{\H_{1}}\ovl{R(\Aos)}
\intertext{hold, which can be further refined and specialized, e.g., to}
\begin{split}
\mylabel{helmrefinedthree}
D(\Ao)&=\ovl{R(\Az)}\oplus_{\H_{1}}N_{0,1}\oplus_{\H_{1}}D(\cAo),\\
D(\Azs)&=D(\cAzs)\oplus_{\H_{1}}N_{0,1}\oplus_{\H_{1}}\ovl{R(\Aos)},\\
D(\Ao)\cap D(\Azs)&=D(\cAzs)\oplus_{\H_{1}}N_{0,1}\oplus_{\H_{1}}D(\cAo).
\end{split}
\end{align}
\end{lem}

\begin{proof}
By \eqref{helmappclone} and the complex properties we see \eqref{helmrefinedzero} and \eqref{helmrefinedone},
yielding directly \eqref{helmrefinedtwo} and \eqref{helmrefinedthree}.
\end{proof}

We observe
\begin{align*}
D(\cAo)
&=D(\Ao)\cap\ovl{R(\Aos)}
\subset D(\Ao)\cap N(\Azs)
\subset D(\Ao)\cap D(\Azs),\\
D(\cAzs)
&=D(\Azs)\cap\ovl{R(\Az)}
\subset D(\Azs)\cap N(\Ao)
\subset D(\Azs)\cap D(\Ao),
\end{align*}
and using the refined Helmholtz type decompositions of Lemma \ref{helmrefined}
as well as the results of Lemma \ref{poincarerange}, Lemma \ref{cptequi}, and Lemma \ref{compemblem},
we immediately see: 

\begin{lem}
\label{compemblem}
The following assertions are equivalent: 
\begin{itemize}
\item[\bf(i)]
$D(\cAz)\cptemb\H_{0}$, $D(\cAo)\cptemb\H_{1}$,
and $N_{0,1}\cptemb\H_{1}$ are compact.
\item[\bf(ii)]
$D(\Ao)\cap D(\Azs)\cptemb\H_{1}$ is compact.
\end{itemize}
In this case, the cohomology group $N_{0,1}$ has finite dimension.
\end{lem}

We summarize:

\begin{theo}
\label{compembtheo}
Let $D(\Ao)\cap D(\Azs)\cptemb\H_{1}$ be compact.
Then $D(\cAz)\cptemb\H_{0}$, $D(\cAo)\cptemb\H_{1}$,
as well as $D(\cAzs)\cptemb\H_{1}$, $D(\cAos)\cptemb\H_{2}$ are compact,
$\dim N_{0,1}<\infty$,
all ranges $R(\Az)$, $R(\Azs)$, and $R(\Ao)$, $R(\Aos)$ are closed, and
the corresponding Friedrichs/Poincar\'e type estimates hold, i.e.
there exists positive constants $c_{\Az},c_{\Ao}$ such that
\begin{align}
\mylabel{poincareAz}
\forall\,z&\in D(\cAz)
&
\norm{z}_{\H_{0}}&\leq c_{\Az}\norm{\Az z}_{\H_{1}},\\
\nonumber
\forall\,x&\in D(\cAzs)
&
\norm{x}_{\H_{1}}&\leq c_{\Az}\norm{\Azs x}_{\H_{0}},\\
\nonumber
\forall\,x&\in D(\cAo)
&
\norm{x}_{\H_{1}}&\leq c_{\Ao}\norm{\Ao x}_{\H_{2}},\\
\nonumber
\forall\,y&\in D(\cAos)
&
\norm{y}_{\H_{2}}&\leq c_{\Ao}\norm{\Aos y}_{\H_{1}}.
\end{align}
Moreover, all refined Helmholtz type decompositions of Lemma \ref{helmrefined} 
hold with closed ranges, especially
\begin{align}
\mylabel{helmrefinedimportant}
D(\Ao)&=R(\cAz)\oplus_{\H_{1}}\big(D(\Ao)\cap N(\Azs)\big).
\end{align}
\end{theo}

\begin{proof}
Apply the latter lemmas and remarks to $\A=\Az$ and $\A=\Ao$.
\end{proof}

\subsection{The $\Azs$-$\Ao$-Lemma}

Let $\Az$ and $\Ao$ be as introduced before satisfying
the complex property \eqref{sequenceprop}, i.e., $\Ao\Az=0$ or $R(\Az)\subset N(\Ao)$. 
In other words, the primal and dual sequences
\begin{align}
\begin{split}
\mylabel{complexdiag}
\begin{CD}
D(\Az)\subset\H_{0} @> \Az >>
D(\Ao)\subset\H_{1} @> \Ao >>
\H_{2},
\end{CD}\\
\begin{CD}
\H_{0} @< \Azs <<
D(\Azs)\subset\H_{1} @< \Aos <<
D(\Aos)\subset\H_{2}
\end{CD}
\end{split}
\end{align}
are Hilbert complexes of closed and densely defined linear operators.
The additional assumption that the ranges $R(\Az)$ and $R(\Ao)$ are closed
\big(and then also the ranges $R(\Azs)$ and $R(\Aos)$\big)
is equivalent to the closedness of the Hilbert complexes.
Moreover, the complexes are exact if and only if $N_{0,1}=\{0\}$.

As our main result, the following generalized global $\div$-$\curl$-lemma holds.

\begin{theo}[$\Azs$-$\Ao$-lemma]
\mylabel{gen-div-rot-lem}
Let $D(\Ao)\cap D(\Azs)\cptemb\H_{1}$ be compact.
Moreover, let $(x_{n})\subset D(\Ao)$ and $(y_{n})\subset D(\Azs)$ 
be two sequences bounded in $D(\Ao)$ and $D(\Azs)$, respectively.
Then there exist $x\in D(\Ao)$ and $y\in D(\Azs)$ as well as subsequences, again denoted by $(x_{n})$ and $(y_{n})$, 
such that
\begin{itemize}
\item
$x_{n}\wto x$ in $D(\Ao)$,
\item
$y_{n}\wto y$ in $D(\Azs)$,
\item
$\scp{x_{n}}{y_{n}}_{\H_{1}}\to\scp{x}{y}_{\H_{1}}$.
\end{itemize}
\end{theo}

\begin{proof}
Note that Theorem \ref{compembtheo} can be applied.
We pick subsequences, again denoted by $(x_{n})$ and $(y_{n})$, 
such that $(x_{n})$ and $(y_{n})$ converge weakly in $D(\Ao)$ and $D(\Azs)$ 
to $x\in D(\Ao)$ and $y\in D(\Azs)$, respectively. 
By \eqref{helmrefinedimportant} we get the orthogonal decomposition 
$$D(\Ao)\ni x_{n}=\Az z_{n}+\tilde{x}_{n},\quad
z_{n}\in D(\cAz),\quad
\tilde{x}_{n}\in D(\Ao)\cap N(\Azs).$$ 
$(z_{n})$ is bounded in $D(\cAz)$ by orthogonality 
and the Friedrichs/Poincar\'e type estimate \eqref{poincareAz}.
$(\tilde{x}_{n})$ is bounded in $D(\Ao)\cap D(\Azs)$ by orthogonality 
and $\Ao\tilde{x}_{n}=\Ao x_{n}$, $\Azs\tilde{x}_{n}=0$.
Using the compact embeddings $D(\cAz)\cptemb\H_{0}$ and 
$D(\Ao)\cap D(\Azs)\cptemb\H_{1}$,
there exist $z\in D(\cAz)$ and $\tilde{x}\in D(\Ao)\cap N(\Azs)$
and we can extract two subsequences, again denoted by $(z_{n})$ and $(\tilde{x}_{n})$, 
such that $z_{n}\rightharpoonup z$ in $D(\Az)$ and $z_{n}\to z$ in $\H_{0}$
as well as $\tilde{x}_{n}\rightharpoonup\tilde{x}$ in $D(\Ao)\cap D(\Azs)$ 
and $\tilde{x}_{n}\to\tilde{x}$ in $\H_{1}$. 
We have $x=\Az z+\tilde{x}$, giving the Helmholtz type decomposition for $x$, 
as, e.g., for all $\varphi\in\H_{1}$
\begin{align*}
\scp{x}{\varphi}_{\H_{1}}
&\ot\scp{x_{n}}{\varphi}_{\H_{1}}
=\scp{\Az z_{n}}{\varphi}_{\H_{1}}
+\scp{\tilde{x}_{n}}{\varphi}_{\H_{1}}
\to\scp{\Az z}{\varphi}_{\H_{1}}
+\scp{\tilde{x}}{\varphi}_{\H_{1}}.
\end{align*}
Finally, we see
\begin{align*}
\scp{x_{n}}{y_{n}}_{\H_{1}}
&=\scp{\Az z_{n}}{y_{n}}_{\H_{1}}
+\scp{\tilde{x}_{n}}{y_{n}}_{\H_{1}}
=\scp{z_{n}}{\Azs y_{n}}_{\H_{0}}
+\scp{\tilde{x}_{n}}{y_{n}}_{\H_{1}}\\
&\to\scp{z}{\Azs y}_{\H_{0}}
+\scp{\tilde{x}}{y}_{\H_{1}}
=\scp{\Az z}{y}_{\H_{1}}
+\scp{\tilde{x}}{y}_{\H_{1}}
=\scp{x}{y}_{\H_{1}},
\end{align*}
completing the proof.
\end{proof}

\subsection{Generalizations of the $\Azs$-$\Ao$-Lemma}
\mylabel{sec-gen-divcurl-lem}

In this section we present and discuss 
some variants of Theorem \ref{gen-div-rot-lem} using weaker assumptions,
which are taken from the nice paper \cite{waurick2018a} of Marcus Waurick.
We start with the following remarks.

\begin{rem}
\mylabel{gen-div-rot-lem-rem-1}
By Lemma \ref{compemblem} the crucial assumption, i.e., $D(\Ao)\cap D(\Azs)\cptemb\H_{1}$ is compact,
holds, if and only if $D(\cAz)\cptemb\H_{0}$, $D(\cAo)\cptemb\H_{1}$ are compact
and $N_{0,1}$ is finite-dimensional. Moreover, as Banach space adjoints we have
\begin{align*}
\H_{0}'&\cptemb D(\cAz)'
&
&\equi
&
D(\cAz)&\cptemb\H_{0}
&
&\equi
&
D(\cAzs)&\cptemb\H_{1}
&
&\equi
&
\H_{1}'&\cptemb D(\cAzs)',
\intertext{and}
\H_{1}'&\cptemb D(\cAo)'
&
&\equi
&
D(\cAo)&\cptemb\H_{1}
&
&\equi
&
D(\cAos)&\cptemb\H_{2}
&
&\equi
&
\H_{2}'&\cptemb D(\cAos)'.
\end{align*}
In particular, the assumption on the compactness of $D(\Ao)\cap D(\Azs)\cptemb\H_{1}$
is equivalent to the assumptions that $\dim N_{0,1}<\infty$ and 
$\H_{0}\cong\H_{0}'\cptemb D(\cAz)'$, $\H_{2}\cong\H_{2}'\cptemb D(\cAos)'$ are compact.
Thus we observe that the assumptions of Theorem \ref{gen-div-rot-lem}
are stronger but closely related to those of \cite[Theorem 2.4]{waurick2018a}.
Recall that by Theorem \ref{compembtheo} both ranges 
$R(\Az)$ and $R(\Ao)$ are closed and that $\dim N_{0,1}<\infty$
if $D(\Ao)\cap D(\Azs)\cptemb\H_{1}$ is compact.
We emphasize that we have provided a different proof
under stronger assumptions, which is from our personal point of view and taste
easier and more canonical.
\end{rem}

Let us discuss the relations to \cite{waurick2018a},
in particular to \cite[Theorem 2.4]{waurick2018a}, in more detail.
First we note that Theorem \ref{gen-div-rot-lem}
is equivalent to \cite[Theorem 2.5]{waurick2018a}
and that the assumptions of Theorem \ref{gen-div-rot-lem}
are stronger but closely related to those of \cite[Theorem 2.4]{waurick2018a}.

A closer inspection of the proof of Theorem \ref{gen-div-rot-lem} shows that
we can deal with slightly weaker assumptions.
For this, let $R(\Az)$ and $R(\Ao)$ be closed 
(which automatically would be implied by the compact embedding $D(\Ao)\cap D(\Azs)\cptemb\H_{1}$,
see Theorem \ref{compembtheo}),
and let $(x_{n})\subset D(\Ao)$ and $(y_{n})\subset D(\Azs)$ 
be two sequences bounded in $\H_{1}$. 
By \eqref{helmrefinedthree} we have
\begin{align}
\label{HelmdecoweakA}
\begin{split}
D(\Ao)\ni x_{n}=\Az z_{n}+\hat{x}_{n}+\Aos w_{n}&\in R(\cAz)\oplus_{\H_{1}}N_{0,1}\oplus_{\H_{1}}D(\cAo),\\
D(\Azs)\ni y_{n}=\Az u_{n}+\hat{y}_{n}+\Aos v_{n}&\in D(\cAzs)\oplus_{\H_{1}}N_{0,1}\oplus_{\H_{1}}R(\cAos),
\end{split}
\end{align}
with $(z_{n})$ and $(v_{n})$ bounded in $D(\cAz)$ and $D(\cAos)$
by Lemma \ref{poincarerange}, respectively.
W.l.o.g. we can assume that
$(z_{n})$ and $(v_{n})$ already converge weakly in $D(\cAz)$ and $D(\cAos)$, respectively.
Orthogonality shows 
\begin{align}
\label{calcweakAstrong}
\begin{split}
\scp{x_{n}}{y_{n}}_{\H_{1}}
&=\scp{\Az z_{n}}{y_{n}}_{\H_{1}}
+\scp{\hat{x}_{n}}{\hat{y}_{n}}_{\H_{1}}
+\scp{x_{n}}{\Aos v_{n}}_{\H_{1}}\\
&=\scp{z_{n}}{\Azs y_{n}}_{\H_{0}}
+\scp{\hat{x}_{n}}{\hat{y}_{n}}_{\H_{1}}
+\scp{\Ao x_{n}}{v_{n}}_{\H_{2}}.
\end{split}
\end{align}
Hence, we observe that after extracting subsequences,
$\big(\scp{x_{n}}{y_{n}}_{\H_{1}}\big)$ converges, provided that
$N_{0,1}$ is finite-dimensional and 
$(\Azs y_{n})$ and $(\Ao x_{n})$ are relatively compact in
$D(\cAz)'$ and $D(\cAos)'$, respectively. 
This is almost the statement of \cite[Theorem 2.4]{waurick2018a},
still with stronger assumptions.

\subsubsection{More Generalizations}
\label{sec:moregen}

The latter idea can be generalized and, indeed, 
in \cite[Theorem 2.4]{waurick2018a} a more general situation is considered
as $(x_{n})\subset D(\Ao)$ and $(y_{n})\subset D(\Azs)$ are not assumed to hold. 
In fact, these conditions are replaced by corresponding canonical distributional versions
making the respective operators continuous on certain natural dual spaces.
For this we need some preliminaries and new notations.

Dual pairs $(\A,\As)$, $(\cA,\cAs)$
of densely defined and closed (unbounded) linear operators
(as discussed in the latter sections)
with domains of definitions $D(\A)$, $D(\cA)$ and $D(\As)$, $D(\cAs)$,
which are Hilbert spaces equipped with the respective graph norms,
and closed ranges $R(\A)=R(\cA)$ and $R(\As)=R(\cAs)$
can also be considered as \underline{bounded} linear operators.
More precisely,
\begin{align*}
\A:D(\A)&\to\H_{2},
&
\As:D(\As)&\to\H_{1},\\
\cA:D(\cA)&\to R(\cA)=R(\A),
&
\cAs:D(\cAs)&\to R(\cAs)=R(\As)
\intertext{are bounded with bounded Banach space adjoints}
\A':\H_{2}'&\to D(\A)',
&
(\As)':\H_{1}'&\to D(\As)',\\
\cA':R(\A)'&\to D(\cA)',
&
(\cAs)':R(\As)'&\to D(\cAs)',
\end{align*}
defined as usual by
\begin{align*}
\A'y'(\varphi)&:=y'(\A\varphi),
&
y'&\in\H_{2}',
&
\varphi&\in D(\A),\\
(\As)'x'(\phi)&:=x'(\As\phi),
&
x'&\in\H_{1}',
&
\phi&\in D(\As),\\
\cA'y'(\varphi)&:=y'(\cA\varphi),
&
y'&\in R(\A)',
&
\varphi&\in D(\cA),\\
(\cAs)'x'(\phi)&:=x'(\cAs\phi),
&
x'&\in R(\As)'
&
\phi&\in D(\cAs).
\end{align*}
Moreover, we introduce the standard Riesz isomorphisms
\begin{align*}
\R_{\H_{n}}:\H_{n}&\to\H_{n}',
&
\R_{R(\A)}:R(\A)&\to R(\A)',
&
\R_{R(\As)}:R(\As)&\to R(\As)'
\end{align*}
by $x\mapsto\scp{\,\cdot\,}{x}_{\H_{n}}$.
Note that the closed ranges are itself Hilbert spaces with the inner products of $\H_{n}$. 
Using the latter operators we define linear extensions of $\A$, $\cA$ and $\As$, $\cAs$ by
\begin{align*}
\widetilde{\A}:=(\As)'\R_{\H_{1}}:\H_{1}&\to D(\As)',
&
\widetilde{\As}:=\A'\R_{\H_{2}}:\H_{2}&\to D(\A)',\\
\widetilde{\cA}:=(\cAs)'\R_{R(\As)}:R(\As)&\to D(\cAs)',
&
\widetilde{\cAs}:=\cA'\R_{R(\A)}:R(\A)&\to D(\cA)',
\end{align*}
with actions given by
\begin{align*}
\widetilde{\A}\,x(\phi)&=(\As)'\R_{\H_{1}}x(\phi)=\R_{\H_{1}}x(\As\phi)=\scp{\As\phi}{x}_{\H_{1}},
&
x&\in\H_{1},
&
\phi&\in D(\As),\\
\widetilde{\cA}\,x(\phi)&=(\cAs)'\R_{R(\As)}x(\phi)=\R_{R(\As)}x(\cAs\phi)=\scp{\As\phi}{x}_{\H_{1}},
&
x&\in R(\As),
&
\phi&\in D(\cAs),\\
\widetilde{\As}\,y(\varphi)&=\A'\R_{\H_{2}}y(\varphi)=\R_{\H_{2}}y(\A\varphi)=\scp{\A\varphi}{y}_{\H_{2}},
&
y&\in\H_{2},
&
\varphi&\in D(\A),\\
\widetilde{\cAs}\,y(\varphi)&=\cA'\R_{R(\A)}y(\varphi)=\R_{R(\A)}y(\cA\varphi)=\scp{\A\varphi}{y}_{\H_{2}},
&
y&\in R(\A),
&
\varphi&\in D(\cA).
\end{align*}
Introducing the canonical embeddings and their adjoints
\begin{align*}
\iota_{D(\A)}:D(\A)&\hookrightarrow\H_{1},
&
\iota_{D(\A)}':\H_{1}'&\hookrightarrow D(\A)',\\
\iota_{D(\cA)}:D(\cA)&\hookrightarrow R(\As),
&
\iota_{D(\cA)}':R(\As)'&\hookrightarrow D(\cA)',\\
\iota_{D(\As)}:D(\As)&\hookrightarrow\H_{2},
&
\iota_{D(\As)}':\H_{2}'&\hookrightarrow D(\As)',\\
\iota_{D(\cAs)}:D(\cAs)&\hookrightarrow R(\A),
&
\iota_{D(\cAs)}':R(\A)'&\hookrightarrow D(\cAs)'
\end{align*}
we emphasize that for all  $x\in D(\A)$ and for all $\phi\in D(\As)$
$$\widetilde{\A}\,x(\phi)
=\scp{\As\phi}{x}_{\H_{1}}
=\scp{\phi}{\A x}_{\H_{2}}
=\scp{\iota_{D(\As)}\phi}{\A x}_{\H_{2}}
=\R_{\H_{2}}\A x(\iota_{D(\As)}\phi)
=\iota_{D(\As)}'\R_{\H_{2}}\A x(\phi)$$
holds and therefore 
$$\widetilde{\A}|_{D(\A)}:=\widetilde{\A}\iota_{D(\A)}=\iota_{D(\As)}'\R_{\H_{2}}\A:D(\A)\to D(\As)'.$$
Thus, in this sense, $\widetilde{\A}$ is indeed an extension of $\A$.
In the same way we see that
$$\widetilde{\cA}|_{D(\cA)}=\iota_{D(\cAs)}'\R_{R(\A)}\cA,\qquad
\widetilde{\As}|_{D(\As)}=\iota_{D(\A)}'\R_{\H_{1}}\As,\qquad
\widetilde{\cAs}|_{D(\cAs)}=\iota_{D(\cA)}'\R_{R(\As)}\cAs$$
are extensions as well.

\begin{lem}[{\cite[Theorem 2.2]{waurick2018a}}]
\mylabel{adjointlemmaAAs}
Let $R(\A)$ be closed.
Then 
\begin{itemize}
\item[\bf(i)]
$\cA$, $(\cAs)'$, $\widetilde{\cA}$ are topological isomorphisms,
\item[\bf(i${}^{*}$)]
$\cAs$, $\cA'$, $\widetilde{\cAs}$ are topological isomorphisms,
\item[\bf(ii)]
$N(\widetilde{\A})=N(\A)$,
\item[\bf(ii${}^{*}$)]
$N(\widetilde{\As})=N(\As)$,
\item[\bf(iii)]
$\A'$, $\widetilde{\As}$ are surjective if and only if $N(\A)=0$,
\item[\bf(iii${}^{*}$)]
$(\As)'$, $\widetilde{\A}$ are surjective if and only if $N(\As)=0$.
\end{itemize}
\end{lem}

\begin{proof}
$\cA$ and $\cAs$ are a topological isomorphisms by the bounded inverse theorem
or the considerations from the previous sections. 
If $\cA'y'=0$ for $y'\in R(\A)'$, then $\cA'y'(z)=y'(\cA z)=0$
for all $z\in D(\cA)$. Hence $y'=0$ on $R(\cA)=R(\A)$, i.e., $y'=0$.
Thus $\cA'$ is injective and so is $\widetilde{\cAs}=\cA'\R_{R(\A)}$
as $\R_{R(\A)}$ is an isomorphism. 
For $f\in D(\cA)'$ we obtain by Riesz' representation theorem
a unique $z\in D(\cA)$ such that
$$\forall\,\varphi\in D(\cA)\qquad
\scp{\A\varphi}{\A z}_{\H_{2}}
=f(\varphi).$$
Note that $\scp{\A\,\cdot\,}{\A\,\cdot\,}_{\H_{2}}$
is an inner product for $D(\cA)$ by Lemma \ref{poincarerange}.
Thus with $y:=\A z\in R(\A)$ we see
$$\forall\,\varphi\in D(\cA)\qquad
f(\varphi)
=\scp{\A\varphi}{y}_{\H_{2}}
=\widetilde{\cAs}\,y(\varphi),$$
i.e., $f=\widetilde{\cAs}\,y$. Hence $\widetilde{\cAs}$ is surjective 
and so is $\cA'=\widetilde{\cAs}\R_{R(\A)}^{-1}$ as $\R_{R(\A)}$ is an isomorphism. 
By the bounded inverse theorem both $\cA'$ and $\widetilde{\cAs}$
are topological isomorphisms.
Analogously we show the assertions for
$(\cAs)'$ and $\widetilde{\cA}$, which shows (i) and (i${}^{*}$).
For (ii) we observe $x\in N(\widetilde{\A})$ if and only if 
$$\forall\,\phi\in D(\As)\qquad
0=\widetilde{\A}\,x(\phi)
=\scp{\As\phi}{x}_{\H_{1}},$$
if and only if $x\in N(\A)$. Similarly we see
$N(\widetilde{\As})=N(\As)$, proving (ii${}^{*}$).
Let $N(\A)=\{0\}$ and $f\in D(\A)'$.
Then $D(\A)=D(\cA)$ and following the argument for $\widetilde{\cAs}$ from above
we obtain $y\in R(\A)\subset\H_{2}$ with $f=\widetilde{\As}\,y$. 
Hence $\widetilde{\As}$ is surjective 
and so is $\A'=\widetilde{\As}\R_{\H_{2}}^{-1}$ as $\R_{\H_{2}}$ is an isomorphism. 
On the other hand, $\widetilde{\As}$ is surjective if and only if $\A'$ is surjective, 
and in this case for any $\varphi\in N(\A)$
we can represent $f:=\iota_{D(\A)}'\R_{\H_{1}}\iota_{N(\A)}\varphi\in D(\A)'$ by
$\widetilde{\As}\,y=f$ with some $y\in\H_{2}$. Hence
\begin{align*}
0=\scp{\A\varphi}{y}_{\H_{2}}
=\widetilde{\As}\,y(\varphi)
=f(\varphi)
=\R_{\H_{1}}\iota_{N(\A)}\varphi(\iota_{D(\A)}\varphi)
=\scp{\iota_{D(\A)}\varphi}{\iota_{N(\A)}\varphi}_{\H_{1}}
=\scp{\varphi}{\varphi}_{\H_{1}},
\end{align*}
showing $N(\A)=\{0\}$, i.e., (iii).
Analogously, we show (iii${}^{*}$) for $(\As)'$ and $\widetilde{\A}$,
completing the proof.
\end{proof}

\begin{rem}
\mylabel{adjointlemmaAAsrem}
Another, even shorter proof using annihilators is possible. It holds
$$N(\cA')=R(\cA)^{\circ}=\{0\},\qquad
R(\cA')=N(\cA)^{\circ}=\{0\}^{\circ}=D(\cA)',$$
the latter by the closed range theorem.
Hence $\cA'$ is a topological isomorphism by the bounded inverse theorem.
The same applies to $(\cAs)'$.
The Riesz mappings are topological isomorphisms,
so are $\widetilde{\cA}$, $\widetilde{\cAs}$. Moreover,
\begin{align*}
R\big(\widetilde{\As}\big)=R(\A')&=N(\A)^{\circ},
&
R(\widetilde{\A})=R\big((\As)'\big)&=N(\As)^{\circ}.
\end{align*}
Note that also $N(\A')=R(\A)^{\circ}$ and $N\big((\As)'\big)=R(\As)^{\circ}$ hold.
\end{rem}

Using Hilbert space adjoints we introduce the canonical embeddings and projections 
\begin{align*}
\iota_{R(\A)}:R(\A)&\to\H_{2},
&
\iota_{R(\A)}^{*}:\H_{2}&\to R(\A),
&
\pi_{R(\A)}:=\iota_{R(\A)}\iota_{R(\A)}^{*}:\H_{2}&\to\H_{2},\\
\iota_{R(\As)}:R(\As)&\to\H_{1},
&
\iota_{R(\As)}^{*}:\H_{1}&\to R(\As),
&
\pi_{R(\As)}:=\iota_{R(\As)}\iota_{R(\As)}^{*}:\H_{1}&\to\H_{1}.
\end{align*}

\begin{rem}
\mylabel{projrem}
Indeed, $\pi_{R(\A)}$ and $\pi_{R(\As)}$ are the corresponding projections.
To see this, let us consider, e.g., $\pi_{R(\A)}$.
For $x\in D(\iota_{R(\A)}^{*})=\H_{2}$ with $\iota_{R(\A)}^{*}x\in R(\A)$ 
and all $\phi\in D(\iota_{R(\A)})=R(\A)$ it holds
$$\scp{\phi}{x}_{\H_{2}}
=\scp{\iota_{R(\A)}\phi}{x}_{\H_{2}}
=\scp{\phi}{\iota_{R(\A)}^{*}x}_{R(\A)}
=\scp{\phi}{\pi_{R(\A)}x}_{\H_{2}}.$$
Hence $\pi_{R(\A)}x\in R(\A)$ and $(1-\pi_{R(\A)})x\in R(\A)^{\bot_{\H_{2}}}$.
Moreover, since $\pi_{R(\A)}x\in D(\iota_{R(\A)}^{*})=\H_{2}$ the latter computation shows
for all $\phi\in R(\A)$
$$\scp{\phi}{x}_{\H_{2}}
=\scp{\phi}{\pi_{R(\A)}x}_{\H_{2}}
=\scp{\phi}{\pi_{R(\A)}\pi_{R(\A)}x}_{\H_{2}},$$
i.e., $\pi_{R(\A)}\pi_{R(\A)}x=\pi_{R(\A)}x$ on $R(\A)$.
Finally, $\pi_{R(\A)}$ is self-adjoint.
\end{rem}

Furthermore, we need
\begin{align*}
\iota_{R(\A)}^{*}\iota_{D(\As)}:D(\As)&\to D(\cAs),
&
(\iota_{R(\A)}^{*}\iota_{D(\As)})':D(\cAs)'&\to D(\cA)',\\
\iota_{R(\As)}^{*}\iota_{D(\A)}:D(\A)&\to D(\cA),
&
(\iota_{R(\As)}^{*}\iota_{D(\A)})':D(\cA)'&\to D(\A)'.
\end{align*}
We also emphasize that for $x\in\H_{1}$ it holds 
$(1-\pi_{R(\As)})x\in R(\As)^{\bot_{\H_{1}}}=N(\A)$ and thus
$$x=\pi_{R(\As)}x+(1-\pi_{R(\As)})x\in R(\As)\oplus_{\H_{1}}N(\A)$$
is the Helmholtz decomposition for $x$. 
Analogously for $y\in\H_{2}$ the Helmholtz decomposition is given by
$$y=\pi_{R(\A)}y+(1-\pi_{R(\A)})y\in R(\A)\oplus_{\H_{2}}N(\As).$$
Hence for $x\in D(\A)$ and $y\in D(\As)$ we identify
\begin{align}
\mylabel{piRADA}
\pi_{R(\As)}x
&=\iota_{R(\As)}^{*}\iota_{D(\A)}x\in D(\cA),
&
\pi_{R(\A)}y
&=\iota_{R(\A)}^{*}\iota_{D(\As)}y\in D(\cAs).
\end{align}

\begin{lem}
\mylabel{adjointlemmaAAspi}
Let $R(\A)$ be closed. Then 
\begin{itemize}
\item[\bf(i)]
$\widetilde{\A}=(\iota_{R(\A)}^{*}\iota_{D(\As)})'\widetilde{\cA}\,\iota_{R(\As)}^{*}$
and
\begin{align*}
|\widetilde{\A}\,x|_{D(\As)'}
&=|\widetilde{\cA}\,\iota_{R(\As)}^{*}x|_{D(\cAs)'},
&
|\widetilde{\A}|_{\H_{1}\to D(\As)'}
&=|\widetilde{\cA}|_{R(\As)\to D(\cAs)'},
\end{align*}
\item[\bf(ii)]
$\widetilde{\As}=(\iota_{R(\As)}^{*}\iota_{D(\A)})'\widetilde{\cAs}\iota_{R(\A)}^{*}$
and 
\begin{align*}
|\widetilde{\As}\,x|_{D(\A)'}
&=|\widetilde{\cAs}\,\iota_{R(\A)}x|_{D(\cA)'}^{*},
&
|\widetilde{\As}|_{\H_{2}\to D(\A)'}
&=|\widetilde{\cAs}|_{R(\A)\to D(\cA)'}.
\end{align*}
\end{itemize}
\end{lem}

\begin{proof}
For $x\in\H_{1}$ and $\phi\in D(\As)$ we have 
$\pi_{R(\A)}\phi=\iota_{R(\A)}^{*}\iota_{D(\As)}\phi\in D(\cAs)$ and
\begin{align*}
\widetilde{\A}\,x(\phi)
=\scp{\As\phi}{x}_{\H_{1}}
&=\scp{\pi_{R(\As)}\As\pi_{R(\A)}\phi}{x}_{\H_{1}}
=\scp{\As\pi_{R(\A)}\phi}{\pi_{R(\As)}x}_{\H_{1}}\\
&=\scp{\cAs\iota_{R(\A)}^{*}\iota_{D(\As)}\phi}{\iota_{R(\As)}^{*}x}_{R(\As)}
=\widetilde{\cA}\,\iota_{R(\As)}^{*}x(\iota_{R(\A)}^{*}\iota_{D(\As)}\phi)\\
&=(\iota_{R(\A)}^{*}\iota_{D(\As)})'\widetilde{\cA}\,\iota_{R(\As)}^{*}x(\phi).
\end{align*}
Moreover, by the latter computations for $x\in\H_{1}$
\begin{align*}
|\widetilde{\A}\,x|_{D(\As)'}
=\sup_{\substack{\phi\in D(\As)\\|\phi|_{D(\As)}\leq1}}\scp{\As\phi}{x}_{\H_{1}}
&=\sup_{\substack{\phi\in D(\As)\\|\phi|_{D(\As)}\leq1}}\scp{\As\pi_{R(\A)}\phi}{\pi_{R(\As)}x}_{\H_{1}}\\
&=\sup_{\substack{\psi\in D(\cAs)\\|\psi|_{D(\As)}\leq1}}\scp{\cAs\psi}{\iota_{R(\As)}^{*}x}_{\H_{1}}
=|\widetilde{\cA}\,\iota_{R(\As)}^{*}x|_{D(\cAs)'}
\end{align*}
and thus 
\begin{align*}
|\widetilde{\A}|_{\H_{1}\to D(\As)'}
=\sup_{\substack{x\in\H_{1}\\|x|_{\H_{1}}\leq1}}|\widetilde{\A}\,x|_{D(\As)'}
=\sup_{\substack{x\in\H_{1}\\|x|_{\H_{1}}\leq1}}|\widetilde{\cA}\,\iota_{R(\As)}^{*}x|_{D(\cAs)'}
=\sup_{\substack{z\in R(\As)\\|z|_{\H_{1}}\leq1}}|\widetilde{\cA}\,z|_{D(\cAs)'}
=|\widetilde{\cA}|_{R(\As)\to D(\cAs)'}.
\end{align*}
The assertions in (ii) follow analogously.
\end{proof}

The next result from \cite{waurick2018a} is crucial for the further considerations.
We give a slightly modified version.

\begin{lem}[{\cite[Corollary 2.6]{waurick2018a}}]
\mylabel{moeppisuperlem}
Let $R(\A)$ be closed. 
\begin{itemize}
\item[\bf(i)]
For $(x_{n})\subset\H_{1}$ the following statements are equivalent:
\begin{itemize}
\item[\bf(i$_{1}$)]
$\big(\widetilde{\A}\,x_{n}\big)$ is relatively compact in $D(\As)'$.
\item[\bf(i$_{2}$)]
$\big(\widetilde{\cA}\,\iota_{R(\As)}^{*}x_{n}\big)$ is relatively compact in $D(\cAs)'$.
\item[\bf(i$_{3}$)]
$(\iota_{R(\As)}^{*}x_{n})$ is relatively compact in $R(\As)$.
\item[\bf(i$_{4}$)]
$(\pi_{R(\As)}x_{n})$ is relatively compact in $\H_{1}$.
\item[\bf(i$_{5}$)]
$(\R_{R(\As)}\iota_{R(\As)}^{*}x_{n})$ is relatively compact in $R(\As)'$.
\end{itemize}
If $x_{n}\wto x\in\H_{1}$ in $\H_{1}$,
then either of the latter conditions (i$_{1}$)-(i$_{5}$)
implies $\iota_{R(\As)}^{*}x_{n}\to\iota_{R(\As)}^{*}x$ in $R(\As)$ 
and $\pi_{R(\As)}x_{n}\to\pi_{R(\As)}x$ in $\H_{1}$.
\item[\bf(ii)]
For $(y_{n})\subset\H_{2}$ the following statements are equivalent:
\begin{itemize}
\item[\bf(ii$_{1}$)]
$\big(\widetilde{\As}\,y_{n}\big)$ is relatively compact in $D(\A)'$.
\item[\bf(ii$_{2}$)]
$\big(\widetilde{\cAs}\,\iota_{R(\A)}^{*}y_{n}\big)$ is relatively compact in $D(\cA)'$.
\item[\bf(ii$_{3}$)]
$(\iota_{R(\A)}^{*}y_{n})$ is relatively compact in $R(\A)$.
\item[\bf(ii$_{4}$)]
$(\pi_{R(\A)}y_{n})$ is relatively compact in $\H_{2}$.
\item[\bf(ii$_{5}$)]
$(\R_{R(\A)}\iota_{R(\A)}^{*}y_{n})$ is relatively compact in $R(\A)'$.
\end{itemize}
If $y_{n}\wto x\in\H_{2}$ in $\H_{2}$,
then either of the latter conditions (ii$_{1}$)-(ii$_{5}$)
implies $\iota_{R(\A)}^{*}y_{n}\to\iota_{R(\A)}^{*}y$ in $R(\A)$ 
and $\pi_{R(\A)}y_{n}\to\pi_{R(\A)}y$ in $\H_{2}$.
\end{itemize}
\end{lem}

\begin{proof}
By Lemma \ref{adjointlemmaAAs} (i)
$\widetilde{\cA}=(\cAs)'\R_{R(\As)}:R(\As)\to D(\cAs)'$ 
is a topological isomorphism.
Hence (i$_{2}$)-(i$_{5}$) are equivalent.
The equivalence of (i$_{1}$) and (i$_{2}$) follows by Lemma \ref{adjointlemmaAAspi} (i).
If $x_{n}\wto x$ in $\H_{1}$,
then $\iota_{R(\As)}^{*}x_{n}\wto\iota_{R(\As)}^{*}x$ in $R(\As)$ 
and $\pi_{R(\As)}x_{n}\wto\pi_{R(\As)}x$ in $\H_{1}$.
By a subsequence argument we see that, e.g., (i$_{3}$)
implies $\iota_{R(\As)}^{*}x_{n}\to\iota_{R(\As)}^{*}x$ in $R(\As)$ 
and hence $\pi_{R(\As)}x_{n}\to\pi_{R(\As)}x$ in $\H_{1}$.
Analogously we show (ii).
\end{proof}

With this latter key observation 
we can prove a general (distributional) $\Azs$-$\Ao$-lemma. 
For this, we introducing two bounded linear operators 
$\Az:D(\Az)\to\H_{1}$, $\Ao:D(\Ao)\to\H_{2}$
satisfying the complex property $\Ao\Az=0$
and recall the linear extensions of $\Ao$, $\cAo$ and $\Azs$, $\cAzs$
\begin{align*}
\widetilde{\Ao}:=(\Aos)'\R_{\H_{1}}:\H_{1}&\to D(\Aos)',
&
\widetilde{\Azs}:=\Az'\R_{\H_{1}}:\H_{1}&\to D(\Az)',\\
\widetilde{\cAo}:=(\cAos)'\R_{R(\Aos)}:R(\Aos)&\to D(\cAos)',
&
\widetilde{\cAzs}:=\cAz'\R_{R(\Az)}:R(\Az)&\to D(\cAz)'.
\end{align*}

\begin{theo}[generalized $\Azs$-$\Ao$-lemma, {\cite[Theorem 2.4]{waurick2018a}}]
\mylabel{gen-div-rot-lem-gen-theo}
Let the ranges $R(\Az)$ and $R(\Ao)$ be closed and let $N_{0,1}$ be finite-dimensional.
Moreover, let $(x_{n}),(y_{n})\subset\H_{1}$ be two
bounded sequences such that
\begin{itemize}
\item
$(\widetilde{\Ao}\,x_{n})$ is relatively compact in $D(\Aos)'$,
\item
$(\widetilde{\Azs}\,y_{n})$ is relatively compact in $D(\Az)'$.
\end{itemize}
Then there exist $x,y\in\H_{1}$ as well as subsequences, again denoted by $(x_{n})$ and $(y_{n})$, 
such that
\begin{itemize}
\item
$x_{n}\wto x$ in $\H_{1}$,
\item
$y_{n}\wto y$ in $\H_{1}$,
\item
$\scp{x_{n}}{y_{n}}_{\H_{1}}\to\scp{x}{y}_{\H_{1}}$.
\end{itemize}
\end{theo}

\begin{rem}
\mylabel{gen-div-rot-lem-gen-theo-rem}
By Lemma \ref{moeppisuperlem} the assumptions 
on the relative compactness can be replaced equivalently by the assumptions that 
$(\widetilde{\cAo}\,\iota_{R(\Aos)}^{*}x_{n})$ is relatively compact in $D(\cAos)'$ and that
$(\widetilde{\cAzs}\,\iota_{R(\Az)}^{*}y_{n})$ is relatively compact in $D(\cAz)'$.
\end{rem}

\begin{proof}[Proof of Theorem \ref{gen-div-rot-lem-gen-theo}]
Let $(x_{n}),(y_{n})\subset\H_{1}$ be two bounded sequences.
W.l.o.g. let $x_{n}\wto x$ and $y_{n}\wto y$ in $\H_{1}$.
By Lemma \ref{moeppisuperlem} 
$\pi_{R(\Aos)}x_{n}\to\pi_{R(\Aos)}x$
and $\pi_{R(\Az)}y_{n}\to\pi_{R(\Az)}y$ in $\H_{1}$.
By Lemma \ref{helmrefined}, in particular \eqref{helmrefinedtwo} (compare to \eqref{HelmdecoweakA}),
we have the Helmholtz decompositions
\begin{align}
\label{HelmdecoweakAtwo}
\begin{split}
x_{n}=\pi_{R(\Az)}x_{n}+\pi_{N_{0,1}}x_{n}+\pi_{R(\Aos)}x_{n}&\in R(\Az)\oplus_{\H_{1}}N_{0,1}\oplus_{\H_{1}}R(\Aos),\\
y_{n}=\pi_{R(\Az)}y_{n}+\pi_{N_{0,1}}y_{n}+\pi_{R(\Aos)}y_{n}&\in R(\Az)\oplus_{\H_{1}}N_{0,1}\oplus_{\H_{1}}R(\Aos),
\end{split}
\end{align}
yielding (compare to \eqref{calcweakAstrong})
\begin{align}
\label{calcweakA}
\scp{x_{n}}{y_{n}}_{\H_{1}}
&=\scp{\pi_{R(\Aos)}x_{n}}{y_{n}}_{\H_{1}}
+\scp{\pi_{N_{0,1}}x_{n}}{y_{n}}_{\H_{1}}
+\scp{x_{n}}{\pi_{R(\Az)}y_{n}}_{\H_{1}}.
\end{align}
Similar to \eqref{HelmdecoweakAtwo} we can decompose $x$ and $y$
and w.l.o.g. we can assume that $\pi_{N_{0,1}}x_{n}\to\pi_{N_{0,1}}x$
as $N_{0,1}$ has finite dimension. Finally it follows
$$\scp{x_{n}}{y_{n}}_{\H_{1}}
\to\scp{\pi_{R(\Aos)}x}{y}_{\H_{1}}
+\scp{\pi_{N_{0,1}}x}{y}_{\H_{1}}
+\scp{x}{\pi_{R(\Az)}y}_{\H_{1}}
=\scp{x}{y}_{\H_{1}},$$
completing the proof.
\end{proof}

Now, we make the connection to Theorem \ref{gen-div-rot-lem}
and show that the assumptions in Theorem \ref{gen-div-rot-lem}
imply those of Theorem \ref{gen-div-rot-lem-gen-theo}.

\begin{lem}[{\cite[Corollary 2.7]{waurick2018a}}]
\mylabel{moeppicptpiconv}
Let either $\A\!:\!D(\A)\subset\H_{1}\to\H_{2}$ be a densely defined and closed linear operator
or $\A\!:\!D(\A)\to\H_{2}$ be a continuous linear operator.
Moreover, let $D(\cA)\cptemb\H_{1}$ be compact.
\begin{itemize}
\item[\bf(i)]
Let $(x_{n})\subset D(\A)$ be bounded in $D(\A)$.
Then $(\pi_{R(\As)}x_{n})$ is relatively compact in $\H_{1}$.
Equivalently, $\big(\widetilde{\A}\,x_{n}\big)$ is relatively compact in $D(\As)'$.
\item[\bf(ii)]
Let $(y_{n})\subset D(\As)$ be bounded in $D(\As)$.
Then$(\pi_{R(\A)}y_{n})$ is relatively compact in $\H_{2}$.
Equivalently, $\big(\widetilde{\As}\,y_{n}\big)$ is relatively compact in $D(\A)'$.
\end{itemize}
\end{lem}

\begin{proof}
By Lemma \ref{cptequi} the compactness of $D(\cA)\cptemb\H_{1}$
yields the closedness of $R(\A)$. Hence Lemma \ref{moeppisuperlem} is applicable.
Let $(x_{n})\subset D(\A)$ be bounded in $D(\A)$.
Then by \eqref{piRADA}, see also \eqref{DacA},
$(\pi_{R(\As)}x_{n})\subset D(\cA)$ is bounded in $D(\cA)$.
Hence it contains a subsequence converging in $\H_{1}$.
Lemma \ref{moeppisuperlem} shows the equivalence to the second relative compactness.
Analogously we prove the assertions in (ii).
\end{proof}

For two linear operators $\Az$ and $\Ao$ as in Lemma \ref{moeppicptpiconv}, i.e.,
bounded or unbounded, densely defined and closed,
satisfying the complex property $\Ao\Az=0$ we obtain the following results.

\begin{lem}
\mylabel{moeppiassequi}
Let $D(\Ao)\cap D(\Azs)\cptemb\H_{1}$ be compact.
Moreover, let $(x_{n})\subset D(\Ao)$ and $(y_{n})\subset D(\Azs)$ 
be two sequences bounded in $D(\Ao)$ and $D(\Azs)$, respectively. 
Then:
\begin{itemize}
\item[\bf(i)]
$(\pi_{R(\Aos)}x_{n})$ is relatively compact in $\H_{1}$.
Equivalently, $\big(\widetilde{\Ao}\,x_{n}\big)$ is relatively compact in $D(\Aos)'$.
\item[\bf(ii)]
$(\pi_{R(\Az)}y_{n})$ is relatively compact in $\H_{1}$.
Equivalently, $\big(\widetilde{\Azs}\,y_{n}\big)$ is relatively compact in $D(\Az)'$.
\item[\bf(iii)]
$(\pi_{N_{0,1}}x_{n})$ and $(\pi_{N_{0,1}}y_{n})$ are relatively compact in $\H_{1}$.
\end{itemize}
\end{lem}

\begin{proof}
By Lemma \ref{compemblem} 
$D(\cAz)\cptemb\H_{0}$, $D(\cAo)\cptemb\H_{1}$, $N_{0,1}\cptemb\H_{1}$
are compact, in particular, $N_{0,1}$ is finite dimensional, showing (iii).
Lemma \ref{moeppicptpiconv} yields (i) and (ii).
\end{proof}

\begin{rem}
\mylabel{moeppiassequirem}
By Lemma \ref{cptequi} and Lemma \ref{compemblem}  
the compactness of $D(\Ao)\cap D(\Azs)\cptemb\H_{1}$
implies the closedness of the ranges $R(\Az)$ and $R(\Ao)$
and the finite dimensionality of $N_{0,1}$.
Thus Lemma \ref{moeppiassequi} shows that the proof of Theorem \ref{gen-div-rot-lem-gen-theo}
provides another and different proof for Theorem \ref{gen-div-rot-lem}.
\end{rem}

The above considerations lead to the following insight, which is interesting on its own right.

\begin{lem}
\mylabel{relcptimplitwo}
Let $R(\Az)$ and $R(\Ao)$ be closed.
For a sequence $(x_{n})\subset\H_{1}$ the following assertions are equivalent:
\begin{itemize}
\item[\bf(i)]
$(x_{n})$ is relatively compact in $\H_{1}$.
\item[\bf(ii)]
$(\pi_{R(\Aos)}x_{n})$, $(\pi_{R(\Az)}x_{n})$, and $(\pi_{N_{0,1}}x_{n})$
are relatively compact in $\H_{1}$.
\item[\bf(iii)]
$\big(\widetilde{\Azs}\,x_{n}\big)$, 
$\big(\widetilde{\Ao}\,x_{n}\big)$,
and $(\pi_{N_{0,1}}x_{n})$ are relatively compact 
in $D(\Az)'$, $D(\Aos)'$, and $\H_{1}$, respectively.
\end{itemize}
Moreover, if $(x_{n})\subset D(\Ao)\cap D(\Azs)$ is bounded in $D(\Ao)\cap D(\Azs)$
and $D(\Ao)\cap D(\Azs)\cptemb\H_{1}$ is compact,
then (i), (ii), and (iii) hold.
\end{lem}

\begin{proof}
By the continuity of the projections
and the Helmholtz decompositions \eqref{HelmdecoweakAtwo}, i.e.,
$$x_{n}=\pi_{R(\Az)}x_{n}+\pi_{N_{0,1}}x_{n}+\pi_{R(\Aos)}x_{n}\in R(\Az)\oplus_{\H_{1}}N_{0,1}\oplus_{\H_{1}}R(\Aos),$$
the relative compactness of $(x_{n})$ in $\H_{1}$ is equivalent to (ii),
which is equivalent to (iii) by Lemma \ref{moeppisuperlem}.
The last assertion follows by definition.
\end{proof}

\section{Applications}
\mylabel{appsec}

Whenever closed Hilbert complexes like \eqref{complexdiag} 
together with the corresponding compact embedding $D(\Ao)\cap D(\Azs)\cptemb\H_{1}$ occur, 
we can apply the general $\Azs$-$\Ao$-lemma, i.e., Theorem \ref{gen-div-rot-lem}.
In three dimensions we typically have three closed and densely defined linear operators
$\Az$, $\Ao$, and $\At$, satisfying the complex properties
$R(\Az)\subset N(\Ao)$ and $R(\Ao)\subset N(\At)$, i.e.,
\begin{align}
\begin{split}
\mylabel{complexdiagAzot}
\begin{CD}
D(\Az)\subset\H_{0} @> \Az >>
D(\Ao)\subset\H_{1} @> \Ao >>
D(\At)\subset\H_{2} @> \At >>
\H_{3},
\end{CD}\\
\begin{CD}
\H_{0} @< \Azs <<
D(\Azs)\subset\H_{1} @< \Aos <<
D(\Aos)\subset\H_{2} @< \Ats <<
D(\Ats)\subset\H_{3},
\end{CD}
\end{split}
\end{align}
together with the crucial compact embeddings
\begin{align}
\mylabel{crucialemb}
D(\Ao)\cap D(\Azs)&\cptemb\H_{1},
&
D(\At)\cap D(\Aos)&\cptemb\H_{2}.
\end{align}
With slightly weaker assumptions we can apply Theorem \ref{gen-div-rot-lem-gen-theo}.

Recalling our general assumptions on the underlying domain from Section \ref{defsec},
throughout this application section $\om$ can be a
\begin{itemize}
\item
weak Lipschitz domain with boundary $\ga$,
\item
weak Lipschitz domain with boundary $\ga$ and weak Lipschitz interfaces $\gat$ and $\gan$,
\item
strong Lipschitz domain with boundary $\ga$,
\item
strong Lipschitz domain with boundary $\ga$ and strong Lipschitz interfaces $\gat$ and $\gan$.
\end{itemize}
We extend this definition to $\om\subset\rN$ or Riemannian manifolds $\om$.

\subsection{The div-rot-Lemma Revisited}
\mylabel{appsecdivrotrev}

Let $\om\subset\rt$.
The first example is given by the classical operators from vector analysis
\begin{align*}
\Az:=\gradc_{\gat}:\hocgatom\subset\ltom&\To\ltepsom;
&
u&\mapsto\na u,\\
\Ao:=\mu^{-1}\rotc_{\gat}:\rcgatom\subset\ltepsom&\To\ltmuom;
&
E&\mapsto\mu^{-1}\rot E,\\
\At:=\divc_{\gat}\mu:\mu^{-1}\dcgatom\subset\ltmuom&\To\ltom;
&
H&\mapsto\div\mu H.
\intertext{$\Az$, $\Ao$, and $\At$ are unbounded, densely defined, and closed linear operators with adjoints}
\Azs=\gradc_{\gat}^{*}=-\divc_{\gan}\eps:\eps^{-1}\dcganom\subset\ltepsom&\To\ltom;
&
H&\mapsto-\div\eps H,\\
\Aos=(\mu^{-1}\rotc_{\gat})^{*}=\eps^{-1}\rotc_{\gan}:\rcganom\subset\ltmuom&\To\ltepsom;
&
E&\mapsto\eps^{-1}\rot E,\\
\Ats=(\divc_{\gat}\mu)^{*}=-\gradc_{\gan}:\hocganom\subset\ltom&\To\ltmuom;
&
u&\mapsto-\na u.
\end{align*}
Here, $\eps,\mu:\om\to\rttt$ are symmetric and uniformly positive definite $\liom$-tensor fields.
Moreover, the Hilbert-Lebesgue space $\ltepsom$ 
is defined as the standard Lebesgue space $\ltom$ but with an equivalent inner product
$\scp{\,\cdot\,}{\,\cdot\,}_{\ltepsom}:=\scpltom{\eps\,\cdot\,}{\,\cdot\,}$.
Analogously we define $\ltmuom$.
The complex properties hold as
\begin{align*}
R(\Az)=\gradc_{\gat}\hocgatom&\subset\rcgatzom=N(\Ao),
&
R(\Aos)=\eps^{-1}\rotc_{\gan}\rcganom&\subset\eps^{-1}\dcganzom=N(\Azs),\\
R(\Ao)=\mu^{-1}\rotc_{\gat}\rcgatom&\subset\mu^{-1}\dcgatzom=N(\At),
&
R(\Ats)=\gradc_{\gan}\hocganom&\subset\rcganzom=N(\Aos).
\end{align*}
Hence, the sequences \eqref{complexdiagAzot} read
$$\begin{CD}
\hocgatom\subset\ltom @> \A_{0}=\gradc_{\gat} >>
\rcgatom\subset\ltepsom @> \A_{1}=\mu^{-1}\rotc_{\gat} >>
\mu^{-1}\dcgatom\subset\ltmuom @> \A_{2}=\divc_{\gat}\mu >>
\ltom,
\end{CD}$$
$$\begin{CD}
\ltom @< \A_{0}^{*}=-\divc_{\gan}\eps <<
\eps^{-1}\dcganom\subset\ltepsom @< \A_{1}^{*}=\eps^{-1}\rotc_{\gan} <<
\rcganom\subset\ltmuom @< \A_{2}^{*}=-\gradc_{\gan} <<
\hocganom\subset\ltom.
\end{CD}$$
These are the well-known Hilbert complexes for electro-magnetics,
which are also known as de Rham complexes.
Typical equations arising from the de Rham complex are
systems of electro-magneto statics, e.g.,
\begin{align*}
\A_{1}E=\mu^{-1}\rotc_{\gat}E&=F,\\
\A_{0}^{*}E=-\divc_{\gan}\eps E&=f,
\end{align*}
or simply the Dirichlet-Neumann Laplacians and $\rot\rot$ systems, e.g.,
\begin{align*}
\A_{0}^{*}\A_{0}u=-\divc_{\gan}\eps\gradc_{\gat}u&=f,
&
\A_{1}^{*}\A_{1}E=\eps^{-1}\rotc_{\gan}\mu^{-1}\rotc_{\gat}E&=F,\\
&&
\A_{0}^{*}E=-\divc_{\gan}\eps E&=f.
\end{align*}
The crucial embeddings \eqref{crucialemb}
are compact by Weck's selection theorem, compare to Lemma \ref{weckstlem}.

\begin{lem}[Weck's selection theorem]
\label{weckstepsmulem}
Let $\om\subset\rt$ be a weak Lipschitz domain with weak Lipschitz interfaces.
Then the embeddings
\begin{align*}
D(\Ao)\cap D(\Azs)&=\rcgatom\cap\eps^{-1}\dcganom\cptemb\ltom,\\
D(\At)\cap D(\Aos)&=\mu^{-1}\dcgatom\cap\rcganom\cptemb\ltom
\end{align*}
are compact.
\end{lem}

Note that by interchanging the boundary conditions and $\eps$, $\mu$
the latter two compact embeddings are equal.
A proof can be found in \cite[Theorem 4.7]{bauerpaulyschomburgmcpweaklip}.
Indeed, Weck's selection theorems are independent of the material law tensors $\eps$ or $\mu$.
Choosing the pair $(\Az,\Ao)$ we get by Theorem \ref{gen-div-rot-lem} the following:

\begin{theo}[global $\div\eps$-$\mu^{-1}\rot$-lemma]
\mylabel{diveps-murot-lem}
Let $\rcgatom\cap\eps^{-1}\dcganom\cptemb\ltom$ be compact.
Moreover, let $(E_{n})\subset\rcgatom$ and $(H_{n})\subset\eps^{-1}\dcganom$ 
be two sequences bounded in $\rom$ and $\eps^{-1}\dom$, respectively.
Then there exist $E\in\rcgatom$ and $H\in\eps^{-1}\dcganom$ as well as subsequences, again denoted by $(E_{n})$ and $(H_{n})$, 
such that
\begin{itemize}
\item
$E_{n}\wto E$ in $\rcgatom$,
\item
$H_{n}\wto H$ in $\eps^{-1}\dcganom$,
\item
$\scp{E_{n}}{H_{n}}_{\ltepsom}\to\scp{E}{H}_{\ltepsom}$.
\end{itemize}
\end{theo}

\begin{rem}
\mylabel{diveps-murot-lem-rem}
We note:
\begin{itemize}
\item[\bf(i)]
Considering $(E_{n})$ and $(\eps H_{n})$ shows that Theorem \ref{diveps-murot-lem} 
is equivalent to the global $\div$-$\rot$-lemma Theorem \ref{div-rot-lem}.
\item[\bf(ii)]
Theorem \ref{diveps-murot-lem} has a corresponding local version similar to 
the local $\div$-$\rot$-lemma Corollary \ref{div-rot-lem-loc} and Remark \ref{div-rot-rem}, 
which holds with no regularity or boundedness assumptions on $\om$.
\end{itemize}
\end{rem}

The generalization given in Theorem \ref{gen-div-rot-lem-gen-theo} reads as follows.

\begin{theo}[generalized/distributional global $\div\eps$-$\mu^{-1}\rot$-lemma]
\mylabel{diveps-murot-lem-gen}
Let $\na\hocgatom$ and $\rot\rcgatom$ be closed and let the Dirichlet-Neumann fields
$\rcgatzom\cap\eps^{-1}\dcganzom$ be finite-dimensional.
Moreover, let $(E_{n}),(H_{n})\subset\ltepsom$ be two bounded sequences such that
\begin{itemize}
\item
$(\widetilde{\mu^{-1}\rotc_{\gat}}\,E_{n})$ is relatively compact in $\rcganom'$,
\item
$(\widetilde{\divc_{\gan}\eps}\,H_{n})$ is relatively compact in $\hocgatom'$.
\end{itemize}
Then there exist $E,H\in\ltepsom$ as well as subsequences, again denoted by $(E_{n})$ and $(H_{n})$, 
such that
\begin{itemize}
\item
$E_{n}\wto E$ in $\ltepsom$,
\item
$H_{n}\wto H$ in $\ltepsom$,
\item
$\scp{E_{n}}{H_{n}}_{\ltepsom}\to\scp{E}{H}_{\ltepsom}$.
\end{itemize}
\end{theo}

\begin{rem}
\mylabel{diveps-murot-lem-remtwo}
We emphasize:
\begin{itemize}
\item[\bf(i)]
By Lemma \ref{weckstepsmulem} and Lemma \ref{moeppiassequirem}, 
Theorem \ref{diveps-murot-lem} and Theorem \ref{diveps-murot-lem-gen} hold
for weak Lipschitz domains $\om\subset\rt$ with weak Lipschitz interfaces.
\item[\bf(ii)]
Choosing the pair $(\Ao,\At)$ we get by Theorem \ref{gen-div-rot-lem}
a variant of Theorem \ref{diveps-murot-lem}, shortly stating, that for bounded sequences 
$(E_{n})\subset\mu^{-1}\dcgatom$ and $(H_{n})\subset\rcganom$ it holds (after picking subsequences)
$\scp{E_{n}}{H_{n}}_{\ltmuom}\to\scp{E}{H}_{\ltmuom}$.
Similarly, we get a variant of Theorem \ref{diveps-murot-lem-gen}.
\end{itemize}
\end{rem}

\subsubsection{The Classical $\div$-$\rot$-Lemma}

The classical $\div$-$\rot$-lemma (or $\div$-$\curl$-lemma)
by Murat \cite{murat1978} and Tartar \cite{tartar1979}
reads as a slightly weaker version of Corollary \ref{introcortwo} (local $\div$-$\curl$-lemma) 
from the introduction and uses only the standard dual space
$$\hmoom:=\hocom'.$$

\begin{theo}[classical $\div$-$\rot$-lemma]
\mylabel{class-div-rot-lem}
Let $\om\subset\rt$ be an open set and
let $(E_{n}),(H_{n})\subset\ltom$ be two sequences bounded in $\ltom$ such that
both $(\widetilde{\rot}\,E_{n})$ and $(\widetilde{\div}\,H_{n})$ are relatively compact in $\hmoom$.
Then there exist $E,H\in\ltom$ as well as subsequences, again denoted by $(E_{n})$ and $(H_{n})$, 
such that the sequence of scalar products $(E_{n}\cdot H_{n})$ converges in the sense of distributions, i.e.,
$$\forall\,\varphi\in\cicom\qquad
\int_{\om}\varphi\,(E_{n}\cdot H_{n})
\to\int_{\om}\varphi\,(E\cdot H).$$
\end{theo}

Here, we recall the linear extensions of $\A$ and $\As$ (tilde-operators) from Section \ref{sec:moregen} 
$$\widetilde{\A}=(\As)'\R_{\H_{1}}:\H_{1}\to D(\As)',\qquad
\widetilde{\As}=\A'\R_{\H_{2}}:\H_{2}\to D(\A)'$$
and consider the bounded linear operators and their adjoints
\begin{align*}
\gradc:\hocom&\To\ltom,
&
-\widetilde{\div}=\gradc{}'\R:\ltom&\To\hocom'=\hmoom,\\
\rotc:\rcom&\To\ltom,
&
\widetilde{\rot}=\rotc{}'\R:\ltom&\To\rcom',\\
\divc:\dcom&\To\ltom,
&
-\widetilde{\grad}=\divc{}'\R:\ltom&\To\dcom',\\
\grad:\hoom&\To\ltom,
&
-\widetilde{\divc}=\grad'\R:\ltom&\To\hoom'=:\hmocom,\\
\rot:\rom&\To\ltom,
&
\widetilde{\rotc}=\rot'\R:\ltom&\To\rom',\\
\div:\dom&\To\ltom,
&
-\widetilde{\gradc}=\div'\R:\ltom&\To\dom',
\end{align*}
where $\R:=\R_{\ltom}:\ltom\to\ltom'$ denotes the (scalar or vector valued) Riesz isomorphism of $\ltom$.
Note that the embeddings
\begin{align*}
\rcom',\dcom'&\subset\hocom'=\hmoom,\\
\rom',\dom'&\subset\hoom'=\hmocom,
&
\hmocom=\hoom'\subset\hocom'=\hmoom
\end{align*}
justify the formulations in Theorem \ref{class-div-rot-lem}.

A typical application of Theorem \ref{class-div-rot-lem} 
in homogenization of partial differential equations is given by the following problem:
Let $(u_{n})\subset\hocom$ be the sequence of unique solutions of the Dirichlet-Laplace problems
$$-\widetilde{\div}\,\Theta_{n}\gradc u_{n}=f\in\hmoom,$$
with some tensor (matrix) fields $\Theta_{n}$ having appropriate properties.
Note that for all $\varphi\in\hocom$ we have the variational formulation
$$f(\varphi)
=\gradc{}'\R\Theta_{n}\gradc u_{n}(\varphi)
=\R\Theta_{n}\gradc u_{n}(\gradc\varphi)
=\scpltom{\gradc\varphi}{\Theta_{n}\gradc u_{n}}.$$
Setting
$$E_{n}:=\gradc u_{n}\in\rczom=N(\Ao)\subset\ltom,\qquad
H_{n}:=\Theta_{n}E_{n}\in\ltom$$
we see
$$\widetilde{\rot}\,E_{n}=\rot E_{n}=0\in\hmoom,\qquad
\widetilde{\div}\,H_{n}=-f\in\hmoom$$
and thus both $(\widetilde{\rot}\,E_{n})$ and $(\widetilde{\div}\,H_{n})$ 
are trivially relatively compact in $\hmoom$ as they are even constant.
Hence Theorem \ref{class-div-rot-lem} yields for all $\varphi\in\cicom$ the convergence of
$$\int_{\om}\varphi\,(E_{n}\cdot H_{n})
=\int_{\om}\varphi\,(\gradc u_{n}\cdot\Theta_{n}\gradc u_{n}).$$
Let us conclude that in view of Theorem \ref{diveps-murot-lem-gen} ($\eps=\mu=\id$)
the proper assumptions for $(E_{n}),(H_{n})\subset\ltom$ in Theorem \ref{class-div-rot-lem}
are given either by (Dirichlet-Laplace)
\begin{itemize}
\item
$(\widetilde{\rotc}\,E_{n})$ 
is relatively compact in $\rom'$,
\item
$(\widetilde{\div}\,H_{n})$ 
is relatively compact in $\hocom'=\hmoom$,
\end{itemize}
or (Neumann-Laplace)
\begin{itemize}
\item
$(\widetilde{\rot}\,E_{n})$ 
is relatively compact in $\rcom'$,
\item
$(\widetilde{\divc}\,H_{n})$ 
is relatively compact in $\hoom'=\hmocom$,
\end{itemize}
additionally to the closedness of
the ranges $\gradc\hocom$, $\grad\hoom$ and $\rotc\rcom$, $\rot\rom$ 
as well as the finite dimension of the Dirichlet fields $\rczom\cap\dzom$
and the Neumann fields $\rzom\cap\dc_{0}(\om)$,
which is a topological property of the underlying domain $\om$,
see \cite{picardharmdiff,picardpotential,picardboundaryelectro}.
Note that Theorem \ref{diveps-murot-lem-gen} implies the stronger convergence
$$\int_{\om}E_{n}\cdot H_{n}
=\scp{E_{n}}{H_{n}}_{\ltom}
\to\scp{E}{H}_{\ltom}.$$

\begin{rem}
\mylabel{gen-div-rot-lem-rem-3}
Let $\om\subset\rt$ be a bounded strong Lipschitz domain with trivial topology. Then
\begin{align*}
\rcom'
=\mathsf{D}^{-1}(\om)
&:=\set{F\in\hmoom}{\widehat{\div}\,F\in\hmoom},\\
\dcom'
=\mathsf{R}^{-1}(\om)
&:=\set{F\in\hmoom}{\widehat{\rot}\,F\in\hmoom}
\end{align*}
hold with equivalent norms,
see \cite{paulyzulehner2018b} or for the two-dimensional analog
\cite{braessbook}.
We conjecture that the duals of $\rom$ and $\dom$ are given by
\begin{align*}
\rom'=\mathring{\mathsf{D}}^{-1}(\om)
&:=\set{F\in\hmocom}{\widehat{\divc}\,F\in\hmocom},\\
\dom'=\mathring{\mathsf{R}}^{-1}(\om)
&:=\set{F\in\hmocom}{\widehat{\rotc}\,F\in\hmocom}
\end{align*}
with equivalent norms.
Here, $\widehat{\div}$ and $\widehat{\rot}$ act as operators from $\hmoom$ to $\hmtom$
and $\widehat{\divc}$ and $\widehat{\rotc}$ act as operators from $\hmocom$ to $\hmtcom$.
\end{rem}

We observe the following.

\begin{lem}
\mylabel{gen-div-rot-lem-rem-3-lem}
Let the assertions in Remark \ref{gen-div-rot-lem-rem-3} hold. 
Then for $E\in\ltom$ and $(E_{n})\subset\ltom$ it holds:
\begin{itemize}
\item[\bf(i)]
$\widehat{\div}\,\widetilde{\rot}\,E=0$
\item[\bf(i')]
$\widetilde{\rot}\,E\in\rcom'\qequi\widetilde{\rot}\,E\in\hmoom$
\item[\bf(i'')]
$(\widetilde{\rot}\,E_{n})$ relatively compact in $\rcom'$
$\qequi$
$(\widetilde{\rot}\,E_{n})$ relatively compact in $\hmoom$
\item[\bf(ii)]
$\widehat{\divc}\,\widetilde{\rotc}\,E=0$
\item[\bf(ii')]
$\widetilde{\rotc}\,E\in\rom'\qequi\widetilde{\rotc}\,E\in\hmocom$
\item[\bf(ii'')]
$(\widetilde{\rotc}\,E_{n})$ relatively compact in $\rom'$
$\qequi$
$(\widetilde{\rotc}\,E_{n})$ relatively compact in $\hmocom$
\end{itemize}
\end{lem}

\begin{proof}
For $F:=\widetilde{\rot}\,E\in\rcom'\subset\hmoom$ 
we have $\widehat{\div}\,F=0\in\hmoom$ as for all $\varphi\in\htc(\om)$
$$-\widehat{\div}\,\widetilde{\rot}\,E(\varphi)
=\rotc{}'\R\,E(\na\varphi)
=\R\,E(\rot\na\varphi)
=0,$$
which shows (i), (i'), (i'') by Remark \ref{gen-div-rot-lem-rem-3}.
Analogously we see (ii), (ii'), (ii'').
\end{proof}

Finally, we obtain a refined version of Theorem \ref{diveps-murot-lem-gen}
in the case of full boundary conditions, compare to Theorem \ref{class-div-rot-lem}.

\begin{theo}[improved classical $\div$-$\rot$-lemma]
\mylabel{class-div-rot-lem-imp}
Let $\om\subset\rt$ be a bounded strong Lipschitz domain with trivial topology.
Moreover, let $(E_{n}),(H_{n})\subset\ltom$ be two bounded sequences such that either
\begin{itemize}
\item
$(\widetilde{\rotc}\,E_{n})$ is relatively compact in $\hmocom$,
\item
$(\widetilde{\div}\,H_{n})$ is relatively compact in $\hmoom$
\end{itemize}
or
\begin{itemize}
\item
$(\widetilde{\rot}\,E_{n})$ is relatively compact in $\hmoom$,
\item
$(\widetilde{\divc}\,H_{n})$ is relatively compact in $\hmocom$.
\end{itemize}
Then there exist $E,H\in\ltom$ as well as subsequences, again denoted by $(E_{n})$ and $(H_{n})$, 
such that
\begin{itemize}
\item
$E_{n}\wto E$ in $\ltom$,
\item
$H_{n}\wto H$ in $\ltom$,
\item
$\scp{E_{n}}{H_{n}}_{\ltom}\to\scp{E}{H}_{\ltom}$.
\end{itemize}
\end{theo}

We emphasize that the assumptions on $\om$ in the latter theorem imply
that $\gradc\hocom$, $\grad\hoom$, $\rotc\rcom$, $\rot\rom$ are closed
and that the Dirichlet fields $\rczom\cap\dzom$
and the Neumann fields $\rzom\cap\dc_{0}(\om)$ are finite-dimensional,
even trivial.

A more detailed discussion with nice results
on the connections to the classical $\div$-$\rot$-lemma
can be found in \cite{waurick2018a}.

\subsection{Generalized Electro-Magnetics}

Let $\om\subset\rN$ or let $\om$ even be a smooth Riemannian manifold 
with Lipschitz boundary $\ga$ (Lipschitz submanifold)
having (interface) Lipschitz submanifolds $\gat$, $\gan$. 
Using the calculus of alternating differential $q$-forms, $q=0,\dots,N$,
we define the exterior derivative $\ed$ and co-derivative $\cd=\pm*\ed*$ 
in the weak sense by
\begin{align*}
\dgen{}{q}{}(\om):=\setb{E\in\lgen{}{2,q}{}(\om)}{\ed E\in\lgen{}{2,q+1}{}(\om)},\quad 
\Delta^{q+1}(\om):=\setb{H\in\lgen{}{2,q+1}{}(\om)}{\cd H\in\lgen{}{2,q}{}(\om)},
\end{align*}
where $\lgen{}{2,q}{}(\om)$ denotes the standard Lebesgue space of square integrable $q$-forms.
To introduce boundary conditions we define 
\begin{align*}
\mathring\ed_{\gat}^{q}:
\dcgat^{q}(\om):=\overline{\cicqom{q}{\gat}}^{\dgen{}{q}{}(\om)}
\subset\lgen{}{2,q}{}(\om)&\To\lgen{}{2,q+1}{}(\om);
&
E&\mapsto\ed E
\intertext{as closure of the classical exterior derivative $\ed$ acting on test $q$-forms.
$\mathring\ed_{\gat}^{q}$ is an unbounded, densely defined, and closed linear operator with adjoint}
(\mathring\ed_{\gat}^{q})^{*}=-\mathring\cd_{\gan}^{q+1}:
\mathring\Delta^{q+1}_{\gan}(\om):=\overline{\cicqom{q+1}{\gan}}^{\Delta^{q+1}(\om)}
\subset\lgen{}{2,q+1}{}(\om)&\To\lgen{}{2,q}{}(\om);
&
H&\mapsto-\cd H.
\end{align*}
Let us introduce
$$\A_{0}:=\mathring\ed_{\gat}^{q-1},\quad
\A_{1}:=\mathring\ed_{\gat}^{q},\qquad
\A_{0}^{*}=-\mathring\cd_{\gan}^{q},\quad
\A_{1}^{*}=-\mathring\cd_{\gan}^{q+1}.$$
The complex properties hold as, e.g.,
\begin{align*}
R(\Az)=\mathring\ed_{\gat}^{q-1}\dcgat^{q-1}(\om)&\subset\dc_{\gat,0}^{q}(\om)=N(\Ao),
&
R(\Aos)=\mathring\cd_{\gan}^{q+1}\mathring\Delta^{q+1}_{\gan}(\om)\subset\mathring\Delta^{q}_{\gan,0}(\om)=N(\Azs)
\end{align*}
by the classical properties $\cd\cd=\pm*\ed\ed*=0$.
Hence, the sequences \eqref{complexdiagAzot} read
$$\begin{CD}
\dcgat^{q-1}(\om)\subset\lgen{}{2,q-1}{}(\om) @> \A_{0}=\mathring\ed_{\gat}^{q-1} >>
\dcgat^{q}(\om)\subset\lgen{}{2,q}{}(\om) @> \A_{1}=\mathring\ed_{\gat}^{q} >>
\lgen{}{2,q+1}{}(\om),
\end{CD}$$
$$\begin{CD}
\lgen{}{2,q-1}{}(\om) @< \A_{0}^{*}=-\mathring\cd_{\gan}^{q} <<
\mathring\Delta^{q}_{\gan}(\om)\subset\lgen{}{2,q}{}(\om) @< \A_{1}^{*}=-\mathring\cd_{\gan}^{q+1} <<
\mathring\Delta^{q+1}_{\gan}(\om)\subset\lgen{}{2,q+1}{}(\om),
\end{CD}$$
which are the well-known Hilbert complexes for generalized electro-magnetics, i.e.,
the de Rham complexes. 
Typical equations arising from the de Rham complex are
systems of generalised electro-magneto statics, e.g.,
\begin{align*}
\A_{1}E=\mathring\ed_{\gat}^{q}E&=F,\\
\A_{0}^{*}E=-\mathring\cd_{\gan}^{q}E&=G,
\end{align*}
or systems of generalized Dirichlet-Neumann Laplacians, e.g.,
\begin{align*}
\A_{1}^{*}\A_{1}E=-\mathring\cd_{\gan}^{q+1}\mathring\ed_{\gat}^{q}E&=F,
&
(\A_{1}^{*}\A_{1}+\A_{0}\A_{0}^{*})E
=-(\mathring\cd_{\gan}^{q+1}\mathring\ed_{\gat}^{q}+\mathring\ed_{\gan}^{q-1}\mathring\cd_{\gat}^{q})E&=F,\\
\A_{0}^{*}E=-\mathring\cd_{\gan}^{q}E&=G.
\end{align*}
The crucial embeddings \eqref{crucialemb}
are compact by (a generalization) Weck's selection theorem, compare to Lemma \ref{weckstlem}.

\begin{lem}[Weck's selection theorem]
\label{weckstqlem}
Let $\om\subset\rN$ be a weak Lipschitz domain with weak Lipschitz interfaces
or even a Riemannian manifold with Lipschitz boundary and Lipschitz interfaces.
Then for all $q$ the embeddings
$$D(\Ao)\cap D(\Azs)=\dcgat^{q}(\om)\cap\mathring\Delta^{q}_{\gan}(\om)\cptemb\ltom$$
are compact.
\end{lem}

A proof can be found in \cite[Theorem 4.9]{bauerpaulyschomburg2017},
see also the fundamental papers of Weck \cite{weckmax} (strong Lipschitz) 
and Picard \cite{picardcomimb} (weak Lipschitz) for full boundary conditions.
Again, Weck's selection theorems are independent of possible material law tensors $\eps$ or $\mu$.
Theorem \ref{gen-div-rot-lem} shows the following result:

\begin{theo}[global $\cd$-$\ed$-lemma]
\mylabel{div-rot-lem-qforms}
Let the embedding $\dcgat^{q}(\om)\cap\mathring\Delta^{q}_{\gan}(\om)\cptemb\ltom$ be compact.
Moreover, let $(E_{n})\subset\dcgat^{q}(\om)$ and $(H_{n})\subset\mathring\Delta^{q}_{\gan}(\om)$
be two sequences bounded in $\dgen{}{q}{}(\om)$ and $\Delta^{q}(\om)$, respectively.
Then there exist $E\in\dcgat^{q}(\om)$ and $H\in\mathring\Delta^{q}_{\gan}(\om)$ 
as well as subsequences, again denoted by $(E_{n})$ and $(H_{n})$, 
such that 
\begin{itemize}
\item
$E_{n}\wto E$ in $\dcgat^{q}(\om)$,
\item
$H_{n}\wto H$ in $\Delta^{q}_{\gan}(\om)$,
\item
$\scp{E_{n}}{H_{n}}_{\lgen{}{2,q}{}(\om)}\to\scp{E}{H}_{\lgen{}{2,q}{}(\om)}$.
\end{itemize}
\end{theo}

\begin{rem}
\mylabel{div-rot-lem-qforms-rem}
We note:
\begin{itemize}
\item[\bf(i)]
For $N=3$ and $q=1$ (or $q=2$) we obtain by Theorem \ref{div-rot-lem-qforms}
again the global $\div$-$\rot$-lemma Theorem \ref{div-rot-lem}.
\item[\bf(ii)]
For $q=0$ (or $q=N$) 
as well as identifying $\mathring\ed_{\gat}^{0}=\gradc_{\gat}$ 
and $\mathring\Delta^{0}_{\gan}(\om)=0$
(or $\mathring\ed_{\gat}^{N}=0$ 
and $\mathring\Delta^{N}_{\gan}(\om)=\gradc_{\gan}$)
we get by Theorem \ref{div-rot-lem-qforms} 
the following trivial (by Rellich's selection theorem) result:
For all bounded sequences 
$(u_{n})\subset\hocgatom$ and $(v_{n})\subset\ltom$
there exist $u\in\hocgatom$ and $v\in\ltom$
as well as subsequences, again denoted by $(u_{n})$ and $(v_{n})$, 
such that $(u_{n})$ and $(v_{n})$ converge weakly in $\hocgatom$ (or $\ltom$)
to $u$ and $v$, respectively, together with the convergence of the inner products
$\scpltom{u_{n}}{v_{n}}\to\scpltom{u}{v}$.
\item[\bf(iii)]
Theorem \ref{div-rot-lem-qforms} has a corresponding local version similar to 
the local $\div$-$\rot$-lemma Corollary \ref{div-rot-lem-loc} and Remark \ref{div-rot-rem}, 
which holds with no regularity or boundedness assumptions on $\om$.
\item[\bf(iv)]
Material law tensors $\eps$ and $\mu$ different from the identities can be handled as well.
\end{itemize}
\end{rem}

The generalization given in Theorem \ref{gen-div-rot-lem-gen-theo} reads as follows.

\begin{theo}[generalized/distributional global $\cd$-$\ed$-lemma]
\mylabel{div-rot-lem-qforms-gen}
Let $\ed\dcgat^{q-1}(\om)$ and $\ed\dcgat^{q}(\om)$ be closed and let the generalized Dirichlet-Neumann fields
$\mathring\dsymbol^{q}_{\gat,0}(\om)\cap\mathring\Delta^{q}_{\gan,0}(\om)$ be finite-dimensional.
Moreover, let $(E_{n}),(H_{n})\subset\lgen{}{2,q}{}(\om)$ be two bounded sequences such that
\begin{itemize}
\item
$(\widetilde{\mathring\ed_{\gat}^{q}}\,E_{n})$ is relatively compact in $\mathring\Delta^{q+1}_{\gan}(\om)'$,
\item
$(\widetilde{\mathring\cd_{\gan}^{q}}\,H_{n})$ is relatively compact in $\dcgat^{q-1}(\om)'$.
\end{itemize}
Then there exist $E,H\in\lgen{}{2,q}{}(\om)$ as well as subsequences, again denoted by $(E_{n})$ and $(H_{n})$, 
such that
\begin{itemize}
\item
$E_{n}\wto E$ in $\lgen{}{2,q}{}(\om)$,
\item
$H_{n}\wto H$ in $\lgen{}{2,q}{}(\om)$,
\item
$\scp{E_{n}}{H_{n}}_{\lgen{}{2,q}{}(\om)}\to\scp{E}{H}_{\lgen{}{2,q}{}(\om)}$.
\end{itemize}
\end{theo}

\begin{rem}
\mylabel{div-rot-lem-qforms-remtwo}
By Lemma \ref{weckstqlem} and Lemma \ref{moeppiassequirem}, 
Theorem \ref{div-rot-lem-qforms} and Theorem \ref{div-rot-lem-qforms-gen} hold
for weak Lipschitz domains $\om\subset\rN$ with weak Lipschitz interfaces
or even for Riemannian manifolds $\om$.
\end{rem}

\subsection{Biharmonic Equation, General Relativity, and Gravitational Waves}

Let $\om\subset\rt$. We introduce
symmetric and deviatoric (trace-free) square integrable tensor fields
in $\lt(\om;\bbS)$ and $\lt(\om;\bbT)$ and 
as closures of the Hessian $\na\na$, and $\Rot$, $\Div$ (row-wise $\rot$, $\div$),
applied to test functions or test tensor fields, the linear operators
\begin{align*}
\Az:=\mathring{\na\na}:\htc(\om):=\overline{\cicom}^{\htom}
\subset\ltom&\To\lt(\om;\bbS);
&
u&\mapsto\na\na u,\\
\Ao:=\mathring{\Rot}_{\bbS}:\rc(\om;\bbS):=\overline{\cic(\om;\bbS)}^{\rom}
\subset\lt(\om;\bbS)&\To\lt(\om;\bbT);
&
S&\mapsto\Rot S,\\
\At:=\mathring{\Div}_{\bbT}:\dc(\om;\bbT):=\overline{\cic(\om;\bbT)}^{\dom}
\subset\lt(\om;\bbT)&\To\ltom;
&
T&\mapsto\Div T.
\intertext{\normalsize $\Az$, $\Ao$, and $\At$ are unbounded, densely defined, and closed linear operators with adjoints}
\Azs=(\mathring{\na\na})^{*}=\div\Div_{\bbS}:
\d\d(\om;\bbS)
\subset\lt(\om;\bbS)&\To\ltom;
&
S&\mapsto\div\Div S,\\
\Aos=\mathring{\Rot}_{\bbS}^{*}=\sym\Rot_{\bbT}:
\rsymbol_{\sym}(\om;\bbT)
\subset\lt(\om;\bbT)&\To\lt(\om;\bbS);
&
T&\mapsto\sym\Rot T,\\
\Ats=\mathring{\Div}_{\bbT}^{*}=-\dev\grad:
\hoom
\subset\ltom&\To\lt(\om;\bbT);
&
v&\mapsto-\dev\na v,
\end{align*}
\normalsize
where $\hoom$, $\htom$ denote the usual Sobolev spaces and
\begin{align*}
\rom&:=\setb{S\in\ltom}{\Rot S\in\ltom},
&
\r(\om;\bbS)&:=\rom\cap\lt(\om;\bbS),\\
\dom&:=\setb{T\in\ltom}{\Div T\in\ltom},
&
\d(\om;\bbT)&:=\dom\cap\lt(\om;\bbT),\\
\d\dom&:=\setb{S\in\ltom}{\div\Div S\in\ltom},
&
\d\d(\om;\bbS)&:=\d\d(\om)\cap\lt(\om;\bbS),\\
\rsymbol_{\sym}(\om)&:=\setb{T\in\ltom}{\sym\Rot T\in\ltom},
&
\rsymbol_{\sym}(\om;\bbT)&:=\rsymbol_{\sym}(\om)\cap\lt(\om;\bbT),
\end{align*}
see \cite{paulyzulehnerbiharmonic} for details.
Note that $u$, $v$, and $S$, $T$ are scalar, vector, and tensor (matrix) fields, respectively.
Moreover, for $S\in\r(\om;\bbS)$ it holds $\Rot S\in\lt(\om;\bbT)$.
The complex properties hold as
\begin{align*}
R(\Az)=\mathring{\na\na}\htc(\om)&\subset\rc_{0}(\om;\bbS)=N(\Ao),\\
R(\Aos)=\sym\Rot_{\bbT}\rsymbol_{\sym}(\om;\bbT)&\subset\d\dz(\om;\bbS)=N(\Azs),\\
R(\Ao)=\mathring{\Rot}_{\bbS}\rc(\om;\bbS)&\subset\dc_{0}(\om;\bbT)=N(\At),\\
R(\Ats)=\dev\grad\hoom&\subset\rsymbol_{\sym,0}(\om;\bbT)=N(\Aos),
\end{align*}
see again \cite{paulyzulehnerbiharmonic}.
The sequences \eqref{complexdiagAzot} read
$$\small\begin{CD}
\htc(\om)\subset\ltom @> \A_{0}=\mathring{\na\na} >>
\rc(\om;\bbS)\subset\lt(\om;\bbS) @> \A_{1}=\mathring{\Rot}_{\bbS} >>
\dc(\om;\bbT)\subset\lt(\om;\bbT) @> \A_{2}=\mathring{\Div}_{\bbT} >>
\ltom,
\end{CD}$$
$$\small\begin{CD}
\ltom @< \A_{0}^{*}=\div\Div_{\bbS} <<
\d\d(\om;\bbS)\subset\lt(\om;\bbS) @< \A_{1}^{*}=\sym\Rot_{\bbT} <<
\rsymbol_{\sym}(\om;\bbT)\subset\lt(\om;\bbT) @< \A_{2}^{*}=-\dev\grad <<
\hoom\subset\ltom.
\end{CD}$$
These are the so-called $\mathrm{Grad}\,\mathrm{grad}$ and $\div\Div$ complexes, appearing, e.g., 
in biharmonic problems or general relativity, see \cite{paulyzulehnerbiharmonic} for details.
Typical equations arising from the $\mathrm{Grad}\,\mathrm{grad}$ complex are
systems of general relativity, e.g.,
\begin{align*}
\A_{1}S=\mathring{\Rot}_{\bbS}S&=F,
&
\A_{2}T=\mathring{\Div}_{\bbT}T&=g,\\
\A_{0}^{*}S=\div\Div_{\bbS}S&=f,
&
\A_{1}^{*}T=\sym\Rot_{\bbT}T&=G,
\end{align*}
or simply biharmonic equations and related second order systems, e.g.,
\begin{align*}
\A_{0}^{*}\A_{0}u=\div\Div_{\bbS}\mathring{\na\na}u&=f,
&
\A_{1}^{*}\A_{1}S=\sym\Rot_{\bbT}\mathring{\Rot}_{\bbS}S&=G,\\
&&
\A_{0}^{*}S=\div\Div_{\bbS}S&=f.
\end{align*}
The crucial embeddings \eqref{crucialemb}
are compact, compare to Lemma \ref{weckstlem}.

\begin{lem}[biharmonic selection theorems]
\label{biharmstlem}
Let $\om\subset\rt$ be a strong Lipschitz domain.
Then the embeddings
\begin{align*}
D(\Ao)\cap D(\Azs)&=\rc(\om;\bbS)\cap\d\d(\om;\bbS)\cptemb\lt(\om;\bbS),\\
D(\At)\cap D(\Aos)&=\dc(\om;\bbT)\cap\rsymbol_{\sym}(\om;\bbT)\cptemb\lt(\om;\bbT),
\end{align*}
are compact.
\end{lem}

A proof can be found in \cite[Lemma 3.22]{paulyzulehnerbiharmonic}.
Again, the biharmonic selection theorems are independent of possible material law tensors $\eps$ or $\mu$.
Choosing the pair $(\Az,\Ao)$ we get by Theorem \ref{gen-div-rot-lem} the following:

\begin{theo}[global $\div\Div$-$\Rot$-$\bbS$-lemma]
\mylabel{div-rot-lem-biharmS}
Let $\rc(\om;\bbS)\cap\d\d(\om;\bbS)\cptemb\lt(\om;\bbS)$ be compact.
Moreover, let $(S_{n})\subset\rc(\om;\bbS)$ and $(T_{n})\subset\d\d(\om;\bbS)$ 
be two sequences bounded in $\rom$ and $\d\dom$, respectively.
Then there exist $S\in\rc(\om;\bbS)$ and $T\in\d\d(\om;\bbS)$ as well as subsequences, 
again denoted by $(S_{n})$ and $(T_{n})$, 
such that
\begin{itemize}
\item
$S_{n}\wto S$ in $\rc(\om;\bbS)$, 
\item
$T_{n}\wto T$ in $\d\d(\om;\bbS)$, 
\item
$\scp{S_{n}}{T_{n}}_{\lt(\om,\bbS)}\to\scp{S}{T}_{\lt(\om,\bbS)}$.
\end{itemize}
\end{theo}

For the pair $(\Ao,\At)$ Theorem \ref{gen-div-rot-lem} implies: 

\begin{theo}[global $\sym\Rot$-$\Div$-$\bbT$-lemma]
\mylabel{div-rot-lem-biharmT}
Let $\dc(\om;\bbT)\cap\rsymbol_{\sym}(\om;\bbT)\cptemb\lt(\om;\bbT)$ be compact.
Moreover, let $(S_{n})\subset\dc(\om;\bbT)$ and $(T_{n})\subset\rsymbol_{\sym}(\om;\bbT)$ 
be two sequences bounded in $\dom$ and $\rsymbol_{\sym}(\om)$, respectively.
Then there exist $S\in\dc(\om;\bbT)$ and $T\in\rsymbol_{\sym}(\om;\bbT)$ as well as subsequences, 
again denoted by $(S_{n})$ and $(T_{n})$, 
such that
\begin{itemize}
\item
$S_{n}\wto S$ in $\dc(\om;\bbT)$, 
\item
$T_{n}\wto T$ in $\rsymbol_{\sym}(\om;\bbT)$, 
\item
$\scp{S_{n}}{T_{n}}_{\lt(\om,\bbT)}\to\scp{S}{T}_{\lt(\om,\bbT)}$.
\end{itemize}
\end{theo}

\begin{rem}
\mylabel{div-rot-lem-biharm-rem}
Material law tensors $\eps$ and $\mu$ different from the identities can be handled as well.
Theorem \ref{div-rot-lem-biharmS} and Theorem \ref{div-rot-lem-biharmT}
have corresponding local versions similar to 
the local $\div$-$\rot$-lemma Corollary \ref{div-rot-lem-loc} and Remark \ref{div-rot-rem}, 
which hold with no regularity or boundedness assumptions on $\om$.
We note that the local version of Theorem \ref{div-rot-lem-biharmS}
is a bit more involved as standard localization techniques 
(multiplication by test functions) fail
due to the second order nature of the Sobolev space $\d\d(\om;\bbS)$.
This additional difficulty can be overcome with the help of
a non-standard Helmholtz type decomposition, see \cite[Lemma 3.21]{paulyzulehnerbiharmonic}
and the proof of \cite[Lemma 3.22]{paulyzulehnerbiharmonic}.
\end{rem}

The generalizations from Theorem \ref{gen-div-rot-lem-gen-theo} read as follows.

\begin{theo}[generalized/distributional global $\div\Div$-$\Rot$-$\bbS$-lemma]
\mylabel{div-rot-lem-biharmS-gen}
Let $\na\na\htc(\om)$ and $\Rot\rc(\om;\bbS)$ be closed and let the generalized Dirichlet-Neumann fields
$\rcz(\om;\bbS)\cap\d\dz(\om;\bbS)$ be finite-dimensional.
Moreover, let $(S_{n}),(T_{n})\subset\lt(\om,\bbS)$ be two bounded sequences such that
\begin{itemize}
\item
$(\widetilde{\mathring{\Rot}_{\bbS}}\,S_{n})$ is relatively compact in $\rsymbol_{\sym}(\om;\bbT)'$,
\item
$(\widetilde{\div\Div_{\bbS}}\,T_{n})$ is relatively compact in $\htc(\om)'=\H^{-2}(\om)$.
\end{itemize}
Then there exist $S,T\in\lt(\om,\bbS)$ as well as subsequences, again denoted by $(S_{n})$ and $(T_{n})$, 
such that
\begin{itemize}
\item
$S_{n}\wto S$ in $\lt(\om,\bbS)$,
\item
$T_{n}\wto T$ in $\lt(\om,\bbS)$,
\item
$\scp{S_{n}}{T_{n}}_{\lt(\om,\bbS)}\to\scp{S}{T}_{\lt(\om,\bbS)}$.
\end{itemize}
\end{theo}

\begin{theo}[generalized/distributional global $\sym\Rot$-$\Div$-$\bbT$-lemma]
\mylabel{div-rot-lem-biharmT-gen}
Let the ranges $\Rot\rc(\om;\bbS)$ and $\Div\dc(\om;\bbT)$ be closed and let the generalized Dirichlet-Neumann fields
$\dc_{0}(\om;\bbT)\cap\rsymbol_{\sym,0}(\om;\bbT)$ be finite-dimensional.
Moreover, let $(S_{n}),(T_{n})\subset\lt(\om,\bbT)$ be two bounded sequences such that
\begin{itemize}
\item
$(\widetilde{\mathring{\Div}_{\bbT}}\,S_{n})$ is relatively compact in $\hoom'=\hmocom$,
\item
$(\widetilde{\sym\Rot_{\bbT}}\,T_{n})$ is relatively compact in $\rc(\om;\bbS)'$.
\end{itemize}
Then there exist $S,T\in\lt(\om,\bbT)$ as well as subsequences, again denoted by $(S_{n})$ and $(T_{n})$, 
such that
\begin{itemize}
\item
$S_{n}\wto S$ in $\lt(\om,\bbT)$,
\item
$T_{n}\wto T$ in $\lt(\om,\bbT)$,
\item
$\scp{S_{n}}{T_{n}}_{\lt(\om,\bbT)}\to\scp{S}{T}_{\lt(\om,\bbT)}$.
\end{itemize}
\end{theo}

\begin{rem}
\mylabel{div-rot-lem-biharm-remtwo}
By Lemma \ref{biharmstlem} and Lemma \ref{moeppiassequirem}, 
Theorem \ref{div-rot-lem-biharmS}, Theorem \ref{div-rot-lem-biharmT}, and 
Theorem \ref{div-rot-lem-biharmS-gen}, Theorem \ref{div-rot-lem-biharmT-gen} 
hold for strong Lipschitz domains $\om\subset\rt$.
\end{rem}

\subsection{Linear Elasticity}

Let $\om\subset\rt$ and let
\begin{align*}
\Az:=\mathring{\sym\grad}:\hocom
\subset\ltom&\To\lt(\om;\bbS);
&
v&\mapsto\sym\na v,\\
\Ao:=\mathring{\Rot\Rot}{}^{\top}_{\bbS}:
\mathring{\r\r}{}^{\top}(\om;\bbS):=\overline{\cic(\om;\bbS)}^{\r\r{}^{\top}(\om)}
\subset\lt(\om;\bbS)&\To\lt(\om;\bbS);
&
S&\mapsto\Rot\Rot^{\top}S,\\
\At:=\mathring{\Div}_{\bbS}:\dc(\om;\bbS):=\overline{\cic(\om;\bbS)}^{\dom}
\subset\lt(\om;\bbS)&\To\ltom;
&
T&\mapsto\Div T.
\intertext{\normalsize $\Az$, $\Ao$, and $\At$ are unbounded, densely defined, and closed linear operators with adjoints}
\Azs=(\mathring{\sym\grad})^{*}=-\Div_{\bbS}:
\d(\om;\bbS)
\subset\lt(\om;\bbS)&\To\ltom;
&
S&\mapsto-\Div S,\\
\Aos=(\mathring{\Rot\Rot}{}^{\top}_{\bbS})^{*}=\Rot\Rot{}^{\top}_{\bbS}:
\r\r{}^{\top}(\om;\bbS)
\subset\lt(\om;\bbS)&\To\lt(\om;\bbS);
&
T&\mapsto\Rot\Rot^{\top}T,\\
\Ats=\mathring{\Div}_{\bbS}^{*}=-\sym\grad:\hoom
\subset\ltom&\To\lt(\om;\bbS);
&
v&\mapsto-\sym\na v,
\end{align*}
where $\d(\om;\bbS):=\dom\cap\lt(\om;\bbS)$ and 
\begin{align*}
\r\r{}^{\top}(\om)&:=\setb{S\in\ltom}{\Rot\Rot^{\top}S\in\ltom},
&
\r\r{}^{\top}(\om;\bbS)&:=\r\r{}^{\top}(\om)\cap\lt(\om;\bbS).
\end{align*}
Moreover, for $S\in\r\r{}^{\top}(\om;\bbS)$ it holds $\Rot\Rot^{\top}S\in\lt(\om;\bbS)$.
Note that $v$ and $S$, $T$ are vector and tensor (matrix) fields, respectively.
The complex properties hold as
\begin{align*}
R(\Az)=\mathring{\sym\grad}\hocom
&\subset\mathring{\r\r}{}^{\top}_{0}(\om;\bbS)=N(\Ao),\\
R(\Aos)=\Rot\Rot{}^{\top}_{\bbS}\r\r{}^{\top}(\om;\bbS)
&\subset\dz(\om;\bbS)=N(\Azs),\\
R(\Ao)=\mathring{\Rot\Rot}{}^{\top}_{\bbS}\mathring{\r\r}{}^{\top}(\om;\bbS)
&\subset\dc_{0}(\om;\bbS)=N(\At),\\
R(\Ats)=\sym\grad\hoom
&\subset\r\r{}^{\top}_{0}(\om;\bbS)=N(\Aos).
\end{align*}
The sequences \eqref{complexdiagAzot} read
$$\small\begin{CD}
\hocom\subset\ltom @> \A_{0}=\mathring{\sym\grad} >>
\mathring{\r\r}{}^{\top}(\om;\bbS)\subset\lt(\om;\bbS) @> \A_{1}=\mathring{\Rot\Rot}{}^{\top}_{\bbS} >>
\dc(\om;\bbS)\subset\lt(\om;\bbS) @> \A_{2}=\mathring{\Div}_{\bbS} >>
\ltom,
\end{CD}$$
$$\small\begin{CD}
\ltom @< \A_{0}^{*}=-\Div_{\bbS} <<
\d(\om;\bbS)\subset\lt(\om;\bbS) @< \A_{1}^{*}=\Rot\Rot{}^{\top}_{\bbS} <<
\r\r{}^{\top}(\om;\bbS)\subset\lt(\om;\bbS) @< \A_{2}^{*}=-\sym\grad <<
\hoom\subset\ltom.
\end{CD}$$
These are the so-called $\Rot\Rot$ complexes, appearing, e.g., 
in linear elasticity, see \cite{paulyzulehnerbiharmonic}.
Typical equations arising from the $\Rot\Rot$ complex are
systems of generalized linear elasticity, e.g.,
\begin{align*}
\A_{1}S=\mathring{\Rot\Rot}{}^{\top}_{\bbS}S&=F,\\
\A_{0}^{*}S=-\Div_{\bbS}S&=f,
\end{align*}
or simply linear elasticity and related fourth order $\Rot\Rot\Rot\Rot$ systems, e.g.,
\begin{align*}
\A_{0}^{*}\A_{0}v=-\Div_{\bbS}\mathring{\sym\grad}v&=f,
&
\A_{1}^{*}\A_{1}S=\Rot\Rot{}^{\top}_{\bbS}\mathring{\Rot\Rot}{}^{\top}_{\bbS}S&=G,\\
&&
\A_{0}^{*}S=-\Div_{\bbS}S&=f.
\end{align*}
The crucial embeddings \eqref{crucialemb}
are compact, compare to Lemma \ref{weckstlem}.

\begin{lem}[elasticity selection theorems]
\label{elastlem}
Let $\om\subset\rt$ be a strong Lipschitz domain.
Then the embeddings
\begin{align*}
D(\Ao)\cap D(\Azs)&=\mathring{\r\r}{}^{\top}(\om;\bbS)\cap\d(\om;\bbS)\cptemb\lt(\om;\bbS),\\
D(\At)\cap D(\Aos)&=\dc(\om;\bbS)\cap\r\r{}^{\top}(\om;\bbS)\cptemb\lt(\om;\bbS),
\end{align*}
are compact.
\end{lem}

A proof can be done by the same techniques showing \cite[Lemma 3.22]{paulyzulehnerbiharmonic},
see \cite{paulyschomburgzulehnerelasticity}.
Again, the elasticity selection theorems are independent of possible material law tensors $\eps$ or $\mu$.
Choosing the pair $(\Az,\Ao)$ we get by Theorem \ref{gen-div-rot-lem} the following:

\begin{theo}[global $\Div$-$\Rot\Rot^{\top}$-$\bbS$-lemma]
\mylabel{div-rot-lem-linelaone}
Let $\mathring{\r\r}{}^{\top}(\om;\bbS)\cap\d(\om;\bbS)\cptemb\lt(\om;\bbS)$ be compact.
Moreover, let $(S_{n})\subset\mathring{\r\r}{}^{\top}(\om;\bbS)$ and $(T_{n})\subset\d(\om;\bbS)$ 
be two sequences bounded in $\r\r{}^{\top}(\om)$ and $\dom$, respectively.
Then there exist $S\in\mathring{\r\r}{}^{\top}(\om;\bbS)$ and $T\in\d(\om;\bbS)$ as well as subsequences, 
again denoted by $(S_{n})$ and $(T_{n})$, 
such that
\begin{itemize}
\item
$S_{n}\wto S$ in $\mathring{\r\r}{}^{\top}(\om;\bbS)$, 
\item
$T_{n}\wto T$ in $\d(\om;\bbS)$, 
\item
$\scp{S_{n}}{T_{n}}_{\lt(\om,\bbS)}\to\scp{S}{T}_{\lt(\om,\bbS)}$.
\end{itemize}
\end{theo}

For the pair $(\Ao,\At)$ we obtain:

\begin{theo}[global $\Rot\Rot^{\top}$-$\Div$-$\bbS$-lemma]
\mylabel{div-rot-lem-linelatwo}
Let $\dc(\om;\bbS)\cap\r\r{}^{\top}(\om;\bbS)\cptemb\lt(\om;\bbS)$ be compact.
Moreover, let $(S_{n})\subset\dc(\om;\bbS)$ and $(T_{n})\subset\r\r{}^{\top}(\om;\bbS)$ 
be two sequences bounded in $\dom$ and $\r\r{}^{\top}(\om)$, respectively.
Then there exist $S\in\dc(\om;\bbS)$ and $T\in\r\r{}^{\top}(\om;\bbS)$ as well as subsequences, 
again denoted by $(S_{n})$ and $(T_{n})$, 
such that
\begin{itemize}
\item
$S_{n}\wto S$ in $\dc(\om;\bbS)$, 
\item
$T_{n}\wto T$ in $\r\r{}^{\top}(\om;\bbS)$, 
\item
$\scp{S_{n}}{T_{n}}_{\lt(\om,\bbS)}\to\scp{S}{T}_{\lt(\om,\bbS)}$.
\end{itemize}
\end{theo}

\begin{rem}
\mylabel{div-rot-lem-linela-rem}
Let us note:
\begin{itemize}
\item[\bf(i)]
The $\Rot\Rot$ complexes of linear elasticity have a strong symmetry.
\item[\bf(ii)]
Theorem \ref{div-rot-lem-linelaone} and Theorem \ref{div-rot-lem-linelatwo} 
are the same results just with interchanged boundary conditions.
\item[\bf(iii)]
Theorem \ref{div-rot-lem-linelaone} and Theorem \ref{div-rot-lem-linelatwo} 
have corresponding local versions similar to 
the local $\div$-$\rot$-lemma Corollary \ref{div-rot-lem-loc} and Remark \ref{div-rot-rem}, 
which hold with no regularity or boundedness assumptions on $\om$.
As in Remark \ref{div-rot-lem-biharm-rem} we note that the local versions 
of Theorem \ref{div-rot-lem-linelaone} and Theorem \ref{div-rot-lem-linelatwo}
are more involved as well, here due to the second order nature 
of the Sobolev spaces $\mathring{\r\r}{}^{\top}(\om;\bbS)$ and $\r\r{}^{\top}(\om;\bbS)$.
A corresponding non-standard Helmholtz type decomposition similar to 
\cite[Lemma 3.21]{paulyzulehnerbiharmonic} is needed to overcome these difficulties.
\item[\bf(iv)]
Material law tensors $\eps$ and $\mu$ different from the identities can be handled as well.
\end{itemize}
\end{rem}

The generalizations in Theorem \ref{gen-div-rot-lem-gen-theo} read as follows.

\begin{theo}[generalized/distributional global $\Div$-$\Rot\Rot^{\top}$-$\bbS$-lemma]
\mylabel{div-rot-lem-linelaone-gen}
Let the ranges $\sym\grad\hocom$ and $\Rot\Rot^{\top}\mathring{\r\r}{}^{\top}(\om;\bbS)$
be closed and let the generalized Dirichlet-Neumann fields
$\mathring{\r\r}{}^{\top}_{0}(\om;\bbS)\cap\dz(\om;\bbS)$ be finite-dimensional.
Moreover, let $(S_{n}),(T_{n})\subset\lt(\om,\bbS)$ be two bounded sequences such that
\begin{itemize}
\item
$(\widetilde{\mathring{\Rot\Rot}{}^{\top}_{\bbS}}\,S_{n})$ is relatively compact in $\r\r{}^{\top}(\om;\bbS)'$,
\item
$(\widetilde{\Div_{\bbS}}\,T_{n})$ is relatively compact in $\hocom'=\hmoom$.
\end{itemize}
Then there exist $S,T\in\lt(\om,\bbS)$ as well as subsequences, again denoted by $(S_{n})$ and $(T_{n})$, 
such that
\begin{itemize}
\item
$S_{n}\wto S$ in $\lt(\om,\bbS)$,
\item
$T_{n}\wto T$ in $\lt(\om,\bbS)$,
\item
$\scp{S_{n}}{T_{n}}_{\lt(\om,\bbS)}\to\scp{S}{T}_{\lt(\om,\bbS)}$.
\end{itemize}
\end{theo}

\begin{theo}[generalized/distributional global $\Rot\Rot^{\top}$-$\Div$-$\bbS$-lemma]
\mylabel{div-rot-lem-linelatwo-gen}
Let $\Rot\Rot^{\top}\mathring{\r\r}{}^{\top}(\om;\bbS)$ and $\Div\dc(\om;\bbS)$
be closed and let the generalized Dirichlet-Neumann fields
$\dc_{0}(\om;\bbS)\cap\r\r{}^{\top}_{0}(\om;\bbS)$ be finite-dimensional.
Moreover, let $(S_{n}),(T_{n})\subset\lt(\om,\bbS)$ be two bounded sequences such that
\begin{itemize}
\item
$(\widetilde{\mathring{\Div}_{\bbS}}\,S_{n})$ is relatively compact in $\hoom'=\hmocom$,
\item
$(\widetilde{\Rot\Rot{}^{\top}_{\bbS}}\,T_{n})$ is relatively compact in $\mathring{\r\r}{}^{\top}(\om;\bbS)'$.
\end{itemize}
Then there exist $S,T\in\lt(\om,\bbS)$ as well as subsequences, again denoted by $(S_{n})$ and $(T_{n})$, 
such that
\begin{itemize}
\item
$S_{n}\wto S$ in $\lt(\om,\bbS)$,
\item
$T_{n}\wto T$ in $\lt(\om,\bbS)$,
\item
$\scp{S_{n}}{T_{n}}_{\lt(\om,\bbS)}\to\scp{S}{T}_{\lt(\om,\bbS)}$.
\end{itemize}
\end{theo}

\begin{rem}
\mylabel{div-rot-lem-linelaone-two-rem}
By Lemma \ref{elastlem} and Lemma \ref{moeppiassequirem}, 
Theorem \ref{div-rot-lem-linelaone}, Theorem \ref{div-rot-lem-linelatwo}, and 
Theorem \ref{div-rot-lem-linelaone-gen}, Theorem \ref{div-rot-lem-linelatwo-gen}
hold for strong Lipschitz domains $\om\subset\rt$.
\end{rem}

\begin{acknow}
The author is grateful to S\"oren Bartels for bringing up the topic of the $\div$-$\curl$-lemma,
and especially to Marcus Waurick for lots of inspiring discussions on the $\div$-$\curl$-lemma 
and for his substantial contributions to the Special Semester at RICAM in Linz late 2016.
\end{acknow}

\bibliographystyle{plain} 
\bibliography{paule}

\end{document}